\documentclass[a4paper]{article}

\usepackage{amsthm,amsmath}
\RequirePackage[colorlinks,citecolor=blue,urlcolor=blue]{hyperref}

\usepackage{amssymb}
\usepackage{mathtools}
\usepackage[T1]{fontenc}
\usepackage[utf8]{inputenc}

\newcommand{\additions}[1]{\iffalse #1 \fi}

\newcommand{\RR}{\ensuremath{\mathbb{R}}}

\newcommand{\NN}{\ensuremath{\mathbb{N}}}

\newcommand{\EE}{\ensuremath{\mathbb{E}}}
\newcommand{\N}{\ensuremath{\mathbf{N}}}

\newcommand{\eps}{\ensuremath{\varepsilon}}

\newcommand{\Tcon}{\ensuremath{T_{\text{con}}}}

\newcommand{\Rlex}{\ensuremath{\mathsf{R}_{\text{lex}}}}
\newcommand{\lex}{\ensuremath{{\text{lex}}}}
\newcommand{\slow}{\ensuremath{\mathtt{slow}}}
\newcommand{\fast}{\ensuremath{\mathtt{fast}}}

\newcommand{\ext}{\ensuremath{\mathtt{ex}}}
\newcommand{\comp}{\operatorname{comp}}

\newcommand{\KP}{\mathsf{KP}}

\makeatletter
\newtheorem*{rep@theorem}{\rep@title}
\newcommand{\newreptheorem}[2]{%
\newenvironment{rep#1}[1]{%
 \def\rep@title{#2 \ref{##1}}%
 \begin{rep@theorem}}%
 {\end{rep@theorem}}}
\makeatother

\newtheorem{theorem}{Theorem}
\newreptheorem{theorem}{Theorem}
\newtheorem{lemma}[theorem]{Lemma}
\newreptheorem{lemma}{Lemma}
\newtheorem{remark}[theorem]{Remark}
\newtheorem{corollary}[theorem]{Corollary}

\begin{document}

\title{Connectivity Thresholds for Bounded Size Rules}


\author{Hafsteinn Einarsson\thanks{ETH Zurich, Institute of Theoretical Computer Science, Zürich, Switzerland} \and    
  Johannes Lengler\footnotemark[1] \and                                                                                
  Frank Mousset\footnotemark[1] \and                                                                                   
  Konstantinos Panagiotou\thanks{University of Munich, Mathematics Institute, München, Germany} \and                   
  Angelika Steger\footnotemark[1]                                                                                      
}

\maketitle

\begin{abstract}
In an \emph{Achlioptas process}, starting with a graph that has $n$ vertices
and no edge, in each round $d \ge 1$ edges are drawn uniformly at random, and
using some rule exactly one of them is chosen and added to the evolving graph. 
For the class of Achlioptas processes we investigate how much
impact the rule has on one of the most basic properties of a graph:
connectivity. Our main results are twofold. First, we study the prominent class
of \emph{bounded size rules}, which select the edge to add according to the
component sizes of its vertices, treating all sizes larger than some constant
equally. For such rules we provide a fine analysis that exposes the limiting
distribution of the number of rounds until the graph gets connected, and we
give a detailed picture of the dynamics of the formation of the single
component from smaller components.  Second, our results allow us to study the
connectivity transition of all Achlioptas processes, in the sense that we
identify a process that accelerates it as much as possible.
\end{abstract}


\section{Introduction and Results}

Over the last decades the so-called ``power of choice'' paradigm received a lot of attention in various fields. Very roughly, the term ``power of choice'' stands for the impact  that an observer  can have on a system even if she may influence it only by very small, local choices. Here we just provide a prototypical example. Suppose that we throw $n$ balls uniformly at random into $n$ bins. Then a classical result asserts that the largest number of balls in a bin, the so-called \emph{maximum load}, is close to $\log n / \log\log n$ with high probability (whp), i.e., with probability tending to one as $n\to \infty$, see for example~\cite{ar:g81}. If we instead distribute the balls one after the other, and we place each ball in the least loaded out of $d \ge 2$ randomly selected bins, then the maximum load becomes whp exponentially smaller, namely $\log\log n/\log d + \Theta(1)$, see~\cite{ar:abku99}.

The paradigm of the power of choice has many applications and was investigated in numerous different situations, cf.~\cite{p:mrs00,inc:m09} for some techniques and results. In this paper we study it in the context of the (online) formation of graphs, where the appearance of edges is driven by some random process. An $\ell$\emph{-Achlioptas process} is a game with a single player, Paul, who is building a graph. The game is played in rounds and in the first round Paul starts with a graph that has $n$ vertices and no edge. In each round, $\ell$ uniformly random vertices $v_1,\ldots,v_{\ell}$ are presented, and Paul can choose one of the edges $\{v_1,v_2\}$, $\{v_3,v_4\},\ldots,\{v_{\ell-1},v_{\ell}\}$ to put into the graph; we assume that $\ell$ is even throughout. This game defines a random sequence $(G_N)_{N\geq 0}$, where $G_N$ is the graph after $N$ rounds of the game.

The most prominent and well-studied instance of an Achlioptas process is when $\ell = 2$ and Paul has actually no choice. This is the classical Erd\H{o}s-R\'enyi random graph process, and we denote by $G_N^{\mathsf{ER}}$ the graph that is created after $N$ edges have been added (where we will always ignore multiple edges and loops). The asymptotic properties of $G_N^{\mathsf{ER}}$ have been studied in depth, and the results have numerous applications in mathematics and computer science. One of the most striking and intensely studied phenomenon is the \emph{percolation transition}, which is also described as the emergence of the giant component~\cite{erdHos1960evolution}. Indeed, if we parametrize $N = tn$, then for $t < 1/2$ the largest component in $G_N^{\mathsf{ER}}$ contains whp $O(\log n)$ vertices, while for $t > 1/2$ there is whp a component with $\Theta(n)$ vertices. From today's perspective, the fine details of the phase transition in $G_N^{\mathsf{ER}}$ are well-understood, see e.g.~\cite{MR1864966,ar:b84,ar:jklp93}.

The classical Erd\H{o}s-R\'enyi process contains in fact a deterministic flavor: 
it does not allow for an observer to influence it. Many applications, however, require (or allow) exactly that. Thus, since the seminal work of Erd\H{o}s and R\'enyi various modifications of  their process have been proposed. Among the most prominent ones are Achlioptas processes that utilize the paradigm of the power of choice. As it turned out, the analysis of such processes is far from trivial and requires many new ideas and techniques. In an Achlioptas process Paul can follow various strategies for selecting the next edge. To facilitate the analysis it makes sense to restrict his power to so-called \emph{size rules}: in each round, Paul bases his decision only on the component sizes of the randomly selected vertices $v_1, \dots, v_\ell$ at the beginning of the current round. For many such size rules, and particularly so-called \emph{bounded size} rules, it is by now established that they also exhibit a percolation transition that shares many qualitative characteristics with the transition in the Erd\H{o}s-R\'enyi process~\cite{Spencer2008,bhamidi2012bounded,bohman2006creating,riordan2011convergence,riordan2011explosive,panagiotou2011explosive,riordan2012achlioptas,bhamidi2012augmented,drmota2013pursuing}.

While the study of the percolation transition has attracted lots of attention, the typical properties of a random graph that is created by an Achlioptas process \emph{after} the transition are far less understood. In particular, there are some results concerning the presence of small subgraphs~\cite{krivelevich2009avoiding,krivelevich2012creating,mutze2011small} or Hamiltonicity~\cite{krivelevich2010hamiltonicity}. However, miraculously one of the most basic properties of a graph -- connectivity -- has been studied only very little, see~\cite{ar:kpXX}, and this is the topic of the present work.

Before we state our results we quickly review what is known for the Erd\H{o}s-R\'enyi process.
For $G_N^{\mathsf{ER}}$ the \emph{connectivity transition} is very well-understood. If we write $\Tcon^{\mathsf{ER}}$ for the smallest $N$ for which $G_N^{\mathsf{ER}}$ is connected, then it is known that whp $\Tcon^{\mathsf{ER}} = (1+o(1))n\log n/2$. Moreover, the fine behavior of $\Tcon^{\mathsf{ER}}$ has been studied; in particular, for any $c \in\mathbb{R}$
\[
  \lim_{n\to\infty}\Pr\left[\Tcon^{\mathsf{ER}} \leq \frac{n\log n+ cn}{2}\right]
	=  \exp\{-e^{-c}\}.
\]
Actually, much more can be said. Let $T_1^{\mathsf{ER}}$ denote the smallest $N$ for which $G_N^{\mathsf{ER}}$ contains no isolated vertex. Then whp $T_1^{\mathsf{ER}} = \Tcon^{\mathsf{ER}}$, i.e., the graph becomes whp connected exactly at the round in which the last isolated vertex disappears. For more details we refer to~\cite{MR1864966} and the references therein.

In this paper we study the fine details of the connectivity transition in a broad class of Achlioptas processes. More specifically, we investigate the prominent class of bounded-size rules, which, informally, are size rules such that all component sizes larger than some absolute bound are treated the same. For these processes we give a simple combinatorial criterion that distinguishes between ``degenerate'' and ``non-degenerate'' rules. We show that every degenerate rule needs in expectation $\Omega(n^2)$ rounds to reach the connectivity transition, while every non-degenerate rule needs $\Theta(n\log n)$ rounds. Moreover, if $\Tcon^{\mathsf{R}}$ is the number of rounds until the graph becomes connected for a rule $\mathsf{R}$, then we determine the expectation and the limiting distribution of $\Tcon^{\mathsf{R}}$, which is always a Gumbel distribution. However, we also discover a surprising phenomenon: while the Erd\H{o}s-R\'enyi process becomes whp connected exactly at the round in which the last isolated vertex disappears, this is not true for general bounded size rules. In particular, depending on the rule, several different component sizes may be involved in a ``race'' to get extinct last, and each one of them has a positive probability, which we determine, of achieving this. We give a natural example of such a rule in Section~\ref{ssec:KP}.

Although our results are concerned with bounded-size rules, they enable us to study the connectivity transition of \emph{all} $\ell$-Achlioptas processes in the following sense. A fundamental question is to identify the processes that \emph{accelerate} as much as possible the connectivity transition. We solve this problem by exhibiting in Section~\ref{ssec:LEX} a specific bounded-size rule that is provably the \emph{fastest} among all $\ell$-Achlioptas processes, and we compute the fine details of its connectivity transition.

In order to illustrate our results let us summarize what they imply for a popular rule, the so-called Bohman-Frieze process, $\mathsf{BF}$ for short. There, $\ell = 4$, and in round $N$ Paul chooses $\{v_1,v_2\}$ if and only if both $v_1$ and $v_2$ are isolated vertices in $G_N^{\mathsf{BF}}$ -- otherwise, he selects $\{v_3,v_4\}$ (as usual, $G_0^{\mathsf{BF}}$ contains $n$ vertices and no edge). The $\mathsf{BF}$ process was among the first ones to be studied in the context of graph formation exploiting the power of choice~\cite{bohman2001avoiding}, and it has received vast attention since then. see~\cite{bhamidi2013aggregation,janson2012phase,sen2013largest} and references therein. If we write $\Tcon^\mathsf{BF}$ for the smallest $N$ for which $G_N^{\mathsf{BF}}$ is connected, then our results, see also Section~\ref{sec:BohmanFrieze}, imply that 
\[
\EE[\Tcon^\mathsf{BF}]
= \frac{n\log n}{2} + \left(\frac\gamma2 - \frac{\log\varphi}{\sqrt{5}}\right)\cdot n + o(n),
\]
where $\gamma = 0.577\dots$ is the Euler-Mascheroni constant and $\varphi = (1 + \sqrt{5})/2$ the golden ratio, and
\[
  \lim_{n\to\infty}
		\Pr\left[\Tcon^\mathsf{BF} \leq \frac{n\log n + cn}{2}\right]
		= \exp\{-\varphi^{-2/\sqrt{5}} e^{-c}\} \quad \text{for all $c\in \mathbb{R}$}.
\]
Moreover, we show that whp $G_N^{\mathsf{BF}}$ gets connected when the last isolated vertex disappears. All these results are a special case of Theorem~\ref{thm:main}, which is the main
result of the paper.

\paragraph{Outline} In the next subsection we introduce formally the processes that we study and formulate the main result. The subsequent sections are devoted to the proofs. In Section~\ref{sec:early} we describe the typical structure of the resulting random graphs when the number $N$ of rounds is linear in the number of vertices, and we give tight bounds for the number of components of a given size. Then, in Section~\ref{sec:late} we prove the main result, Theorem~\ref{thm:main}. Section~\ref{sec:degenerate} is devoted to the study of degenerate rules, and the paper closes with some particular examples.

\subsection{\texorpdfstring{$(K,\ell)$}{(K,l)}--rules}\label{sec:rules}

In this paper we study a broad class of random graph processes that in particular include all bounded size rules treated in~\cite{Spencer2008}. We use the conventions $\NN = \{1,2,3,\ldots\}$, $\NN_0 = \{0,1,2,\ldots\}$, and $[m] = \{1,2,\ldots,m\}$ for $m\in \NN$. Let $K, \ell \in \NN$, with $\ell$ even. Let $S_K = \{1,2,\ldots,K,\omega\}$, where $\omega$ stands (informally) for ``larger than $K$''. A $(K,\ell)$-\emph{rule} is a mapping
\[
	\mathsf{R}: S_K^{\ell}\to [\ell/2].
\]
Any such mapping defines naturally a random graph process as follows. For a given graph $G$ we write $\tilde{c}(G;\, v)$ for the number of vertices in the component containing $v$ in $G$. Moreover, set 
\[
	c_K(G;\, v) =
	\begin{cases}
		\tilde{c}(G;\, v) , & \text{if } \tilde{c}(G;\, v) \le K, \\ 
		\omega, &\text{otherwise.}
	\end{cases}
\]
In the following we will often omit the subscript $K$ and the reference to $G$ whenever they are obvious from the context. With this notation, the $\mathsf{R}$-random graph process (or $\mathsf{R}$-process for short) with $n$ vertices is defined as follows. Unless otherwise stated, we begin with $G_0^{\mathsf{R}}$ being the graph with vertex set $[n]$ and no edge. $G_N^{\mathsf{R}}$ is then obtained by choosing independently and uniformly at random $\ell$ vertices $v_1, \dots, v_\ell$ and adding the edge $\{v_{2i-1}, v_{2i}\}$ to $G_{N-1}^{\mathsf{R}}$, where  $i = \mathsf{R}(c_K(v_1), \dots, c_K(v_\ell))$. In words, given the vector of the (truncated) sizes of the components that contain the $v_i$'s, $\mathsf{R}$ determines which of the $\ell/2$ edges determined by the $v_i$'s is to be included into $G_{N-1}^{\mathsf{R}}$.

Note that we do not require $\mathsf{R}$ to be symmetric, e.g.\! we allow for example $\mathsf{R}(1,2,2,\dots,2) \neq \mathsf{R}(2,1,2,\dots,2)$.  Since it is possible that all $\ell/2$ edges in a round are identical, it is clear that the graph will become connected almost surely after a finite number of rounds. The question, of course, is \emph{how quickly} this will take place. In Theorem~\ref{thm:main} and Theorem~\ref{thm:degenerate} we answer this question for a broad class of rules. In order to formulate both theorems, we first introduce some more notation that will be used througout this article.

For a given $(K,\ell)$-rule $\mathsf{R}$ we write $\Tcon^{\mathsf{R}}(n)$ for the smallest $N \in \mathbb{N}_0 \cup\{\infty\}$ such that $G_N^{\mathsf{R}}$ (where $G_0^\mathsf{R}$ has $n$ vertices) is connected. We will usually drop the dependence on $n$, unless it is necessary to make it explicit. For any $k\geq 1$, a \emph{$k$-component} of a graph is a component with $k$ vertices, a \emph{small component} is a component with at most $K$ vertices, and an \emph{$\omega$-component} is a component with more than $K$ vertices. Given $\mu, \nu \in S_K$ let $C_{\mu,\nu} = C_{\mu,\nu}(\mathsf{R})$ be the set of all component size vectors for which a $\mu$-component and a $\nu$-component are connected by an edge in a step of the $\mathsf{R}$-process. More formally,
\[
	C_{\mu,\nu}(\mathsf{R})
	:=
	\left\{s = (s_1, \dots, s_\ell)\in S_K^\ell \mid \{s_{2i-1}, s_{2i}\} = \{\mu,\nu\}, \text{where } i = \mathsf{R}(s)\right\}.
\]
Note that $C_{\mu, \nu} = C_{\nu,\mu}$. For $1\leq k\leq K$ we call
\[ 
	\ext_k(\mathsf{R})
	:=
	k \left|\left\{ s \in C_{k,\omega}(\mathsf{R}) \mid \exists i \in [\ell]: s_i = k \text{ and } s_j = \omega \text{ for all } j \in [\ell]\setminus \{i\}\right\}\right|
\]	
the \emph{extinction rate} for size $k$.  The role of this parameter will become clear at a later point of the analysis. Informally it has the following meaning. Let us consider the $\mathsf{R}$-process at a rather late point $N$ in time, where $G_{N}^{\mathsf{R}}$ is \emph{almost} connected. It is then plausible to assume that $G_{N}^{\mathsf{R}}$ typically consists of one huge component that contains almost all vertices, and all other vertices are in constant-sized components; this is for example the situation in the Erd\H{o}s-R\'enyi process, see e.g.~\cite{RandomGraphsBook,MR1864966}. Then, if we select uniformly at random $\ell$ vertices, then most likely they will all be part of the huge component. However, now and then we will also select a vertex in a small component, say with $k$ vertices, and then the most likely event is that we select \emph{exactly one} such vertex. So, the observed component size vector will look like $(\omega, \dots, k, \dots, \omega)$ with the ``$k$'' at a random position. Whether we actually connect the component of size $k$ with the large component depends on whether this component vector belongs to $C_{k,\omega}(\mathsf{R})$ or not. In other words, the speed with which components of size $k$ disappear depends on the number of such vectors in $C_{k,\omega}(\mathsf{R})$. This explains the second factor in the definition of $\ext_k(\mathsf{R})$. The first factor stems from the fact that a component of size $k$ has $k$ vertices that can be chosen in order to select this component. Indeed, as we will see in the subsequent proof, the smaller $\ext_k(\mathsf{R})$, the later components of size $k$ will disappear in the $\mathsf{R}$-process. We also let
\[
	\ext(\mathsf{R}) := \min_{1\leq k\leq K}{\ext_k(\mathsf{R})}
\]
be the \emph{total extinction rate} of $\mathsf{R}$, and we set
\[
	\slow(\mathsf{R})
		:= \{k\in [K] \mid \ext_k(\mathsf{R}) = \ext(\mathsf{R})\}
	\quad \text{and} \quad
	\fast(\mathsf{R}) := [K] \setminus \slow(\mathsf{R})
\]
the sets of \emph{slow indices} and \emph{fast indices}, respectively. As already mentioned, we will see later in the proof that the main ``obstacles'' that delay the point in time at which $G_N^{\mathsf{R}}$ becomes connected are the $k$-components, where $k \in \slow(\mathsf{R})$. For example, going back to the Erd\H{o}s-R\'enyi case, $\slow(\mathsf{ER}) = \{1\}$, and indeed, the graph becomes connected whp in the round where the last isolated vertex disappears. In contrast, if for some rule $\mathsf{R}$ we have $|\slow(\mathsf{R})| \ge 2$,  then shortly before the connectivity transition, the graph $G_N^\mathsf{R}$ may contain $k$-components for \emph{every} $k\in\slow(\mathsf{R})$, and each one of them has a positive probability of being the last one to disappear; see Section~\ref{ssec:KP} for a natural example.

In our study of the distribution of $\Tcon^{\mathsf{R}}$ it turns out that the value of the total extinction rate essentially determines the point in time where the $\mathsf{R}$-process gets connected, which is whp $(1+o(1))n\log n / \ext(\mathsf{R})$ if $\ext(\mathsf{R}) > 0$. This already shows that the case $\ext(\mathsf{R}) = 0$ is special, and we call a rule \emph{degenerate} if $\ext(\mathsf{R}) = 0$ and \emph{non-degenerate} otherwise.

The main results of this paper are summarized in the following theorem, which asserts that all non-degenerate rules belong to the same ``universality class'': with respect to the connectivity transition, the limiting distribution is always a Gumbel distribution, and the expected value of $\Tcon^\mathsf{R}$ equals $(n \log n + dn)/\ext(\mathsf{R}) + o(n)$ for some $d = d(\mathsf{R})$. To the best of our knownledge the latter statement has not previously been shown even for the $\mathsf{ER}$-process. Finally, only a finite set of component sizes provides the main 'obstacle' for the graph becoming connected.
\begin{theorem}\label{thm:main}
Let $K,\ell \in \mathbb{N}$ and let $\mathsf{R}$ be a non-degenerate $(K,\ell)$-rule such that~$\ext(\mathsf{R})<2K+2$. For $1 \le k \le K$ let $Y_k(N)$ denote the number of vertices in $k$-components in $G_{N}^{\mathsf{R}}$. Moreover, for each $k \in \slow(\mathsf{R})$ there exists\footnote{A formula for $d_k(\mathsf{R})$ is given in Lemma~\ref{lem:limits}. Concrete values of $d_k(\mathsf{R})$ for some rules can be found in the last section.} a constant $d_k = d_k(\mathsf{R})$ such that the following statements are true.
\begin{enumerate}
\item[(a)] For any $c \in \mathbb{R}$, whp for all $N \geq (n\log
n+cn)/\ext(\mathsf{R})$ we have for all $k\in \fast(\mathsf{R})$ that $Y_k(N) =
0$, and there is only one component with more than $K$ vertices in $G_N^{\mathsf{R}}$.
\item[(b)] For any $c \in \mathbb{R}$,
\[
\lim_{n\to\infty}\Pr\left[\Tcon^{\mathsf{R}} \leq \frac{n\log n + cn}{\ext(\mathsf{R})}\right] =  \prod_{k\in \slow(\mathsf{R})}e^{-d_ke^{-c}}.
\]
\item[(c)] Let $\gamma = 0.577...$ be the Euler-Mascheroni constant, and let $c_0 := \log\left(\sum_{k\in \slow(\mathsf{R})}d_k\right)$. Then
\[
\EE[\Tcon^{\mathsf{R}}] = \frac{n\log n+\gamma n +c_0n}{\ext(\mathsf{R})} + o(n).
\]
\item[(d)] For $k\in [K]$, let $T_k^\mathsf{R} := \min\{T \mid \forall N \geq T: Y_k(N) = 0\}$ be the time at which the last $k$-component vanishes. Then $\Pr[T_k^{\mathsf{R}}=\Tcon^{\mathsf{R}}]
  \stackrel{n\to\infty}{\longrightarrow} 0$ for $k \in \fast(\mathsf{R})$, and
  for $k \in \slow(\mathsf{R})$,
\[
  \Pr[T_k^\mathsf{R}=\Tcon^{\mathsf{R}}] \stackrel{n\to\infty}{\longrightarrow} \frac{d_k}{\sum_{i\in\slow(\mathsf{R})}d_i}.
\]
\end{enumerate}
\end{theorem}
\noindent For a better understanding of the theorem we give two remarks.
\begin{remark}\label{rem:extension}
\begin{enumerate}
\item The theorem is in general \emph{not} true if $\ext(\mathsf{R})\geq 2K+2$. However, whenever $K'>K$ then every $(K,\ell)$-rule $\mathsf{R}$ is naturally also a $(K',\ell)$-rule $\mathsf{R}'$, with extinction speeds $\ext_k(\mathsf{R}') = \ext_k(\mathsf{R})$ for $1\leq k\leq K$ and $\ext_k(\mathsf{R}') = 2k$ for $K<k\leq K'$. More precisely, let
\[
\operatorname{trunc}: S_{K'} \to S_K\, , \quad \operatorname{trunc}(k) = \begin{cases} k, & \text{if $k \leq K$}\\ \omega, & \text{otherwise} \end{cases}.
\]
Then the $(K',\ell)$-rule $\mathsf{R}'$ is defined by $$\mathsf{R}'(s_1,\ldots, s_\ell) = \mathsf{R}(\operatorname{trunc}(s_1), \ldots, \operatorname{trunc}(s_\ell)).$$

In particular, if we have a $(K,\ell)$-rule $\mathsf{R}$ for which $\ext(\mathsf{R})\geq 2K+2$, then we can as well express it as a $(\ell/2\cdot K,\ell)$-rule $\mathsf{R}'$, and one checks immediately that $\ext(\mathsf{R}') \leq \ell\cdot K$. In this way, Theorem~\ref{thm:main} is applicable to every bounded size rule.
\item The proof of the theorem will also imply that for all $k\in \slow(\mathsf{R})$,
\begin{align*}
\lim_{n\to\infty}\Pr\left[Y_k\left(\left\lfloor\frac{n\log n +cn}{\ext(\mathsf{R})}\right\rfloor\right) =0\right] & =  e^{-d_ke^{-c}}. \end{align*}
Given these equations, the statement in the theorem shows an ``independence in the limit'' of the variables $Y_k$ in the following sense. Since by Theorem~\ref{thm:main} (a) whp there is only one component with more than $K$ vertices for $N = (n\log n +cn)/\ext(\mathsf{R})$, the graph $G_N^{\mathsf{R}}$ is connected if and only if $Y_k(N)=0$ for all $1\leq k\leq K$. Hence, Theorem~\ref{thm:main} (b) can also be stated as ``$\lim_{n\to \infty}\Pr[Y_k(N) = 0 \text{ for all } 1\le k\le K] = \lim_{n\to \infty}\prod_{1\leq k\leq K} \Pr[Y_k(N)=0]$''.
\end{enumerate}
\end{remark}

\noindent Theorem~\ref{thm:main} only speaks about non-degenerate rules. Degenerate rules $\mathsf{R}$ have the unpleasant property that $\Tcon^{\mathsf{R}}$ can be very large, as it is possible that components of a given fixed size $\le K$ are never connected to other components unless the rule has no choice. In particular, assume that $\ext_k(\mathsf{R}) = 0$ for some $1\le k\le K$, and that there are only two components left in $G_N^{\mathsf{R}}$: a component with $n-k$ vertices and a $k$-component. Then the $k$-component will not be connected to the big component unless at least two of the randomly selected vertices in the current round belong to the $k$-component. Since the probability that this happens is in $O(n^{-2})$ we will need to wait an expected quadratic number of rounds until the graph gets connected. In fact, a similar situation always occurs with non-negligible probability, which is the reason for the following theorem. The proof can be found in Section~\ref{sec:degenerate}.
\begin{theorem}\label{thm:degenerate}
Let $K,\ell \in \mathbb{N}$ and let $\mathsf{R}$ be a degenerate $(K,\ell)$-rule. Then $\EE[\Tcon^\mathsf{R}] = \Omega(n^2)$. 
\end{theorem}

\noindent Note that for certain rules $\Tcon^\mathsf{R}$ can be even larger than $n^2$. Consider for example a $(1,\ell)$-rule that does not take any $1$-component unless forced to, i.e., the rule chooses $(\omega,\omega)$-edges whenever such an edge is available. In the proof of Theorem~\ref{thm:degenerate} we will show that whp there is a situation where only one or two isolated vertices remain. These last vertices will only be collected if in every edge there is at least one isolated vertex. This will eventually happen since we allow a vertex to appear several times in the same round. However, the probability of this event is $O(n^{-\ell/2})$, and thus $\EE[\Tcon^\mathsf{R}] = \Omega(n^{\ell/2})$.

\subsection{Further Terminology and Prerequisites}
\label{ssec:terminology}

For a graph $G$ and an induced subgraph $C$ of $G$ we write $C \in \comp_k(G)$ if $C$ is a $k$-component of~$G$. We say that an event $\mathcal{E} = \mathcal{E}(n)$ holds with high probability (whp)  if $\Pr[\mathcal{E}(n)] \to 1$ for $n\to \infty$. For technical reasons, we will need most statements to hold with probability $1-o(1/\log n)$, and we will say that $\mathcal{E}$ holds \emph{with log-high probability} (wlhp) if $\Pr[\mathcal{E}(n)] \geq 1- o(1/\log n)$ for $n\to \infty$. Without further reference we will use for $x\in[0,1]$ the well-known bounds
\[
	(1-x)^n = 1-nx+O(n^2x^2)
	\quad \text{and} \quad
	1-x = e^{-x + \Theta(x^2)}.
\]
In several proofs we will also exploit the following version of the \emph{Chernoff bounds}, see e.g.~\cite{RandomGraphsBook}.
\begin{lemma}
\label{lemma:Chernoff}
Let $X_1,\ldots,X_n$ be independent Bernoulli variables such that $\Pr[X_i = 1] = p$ and $\Pr[X_i = 0] = 1-p$ for all $1\leq i \leq n$, and let $X = \sum_{i=1}^n X_i$. Then for every $0 \leq \delta \leq 1$,
\[
\Pr[X \geq (1+\delta)np] \leq e^{-{\delta^2}np/3} \qquad \text{and}\qquad \Pr[X \leq (1-\delta)np] \leq e^{-{\delta^2}np/3}.
\]
and
\[
\Pr[X \geq t] \leq 2^{-t} \qquad \text{for all $t\ge 2enp$}.
\]
\end{lemma}


\section{Early Stages of the \texorpdfstring{$\mathsf{R}$}{R}--process}
\label{sec:early}

Let $\mathsf{R}$ be a $(K,\ell)$-rule. In this section we will prove several key lemmas that describe the typical structure of $G_N^\mathsf{R}$ when $N$ is proportional to the number $n$ of vertices. For $k\in S_K$, let the random variable $Y_k^{\mathsf{R}}(N) = Y_k(N)$ denote the number of vertices in $k$-components in $G_N^{\mathsf{R}} = G_N$. Note that $\sum_{k\in S_K} Y_k(N) = n$ for all $N\geq 0$. We will show in Lemma~\ref{lemma:linearregime} that for an appropriate range of $N$, $Y_k(N) = (1+o(1)) \cdot z_k(N/n)\cdot n$, where the $z_k$'s are the unique solution of a specific system of differential equations~\eqref{eq:diffeq3}. To this end, we will use a version of Wormald's method~\cite{Wormald1999}. The argument for establishing the typical trajectory of the $Y_k$'s on the basis of differential equations is rather standard. However, the main contribution of this section is to study in detail the \emph{analytic} properties of the solution of the system~\eqref{eq:diffeq3}, and in particular the case where $N/n$ gets large, see Lemma~\ref{lem:propdiffeq} and~\ref{lem:limits}. These results will be important ingredients in forthcoming arguments.

Let us begin with specifying the system of differential equations. For $s = (s_1,\ldots,s_{\ell}) \in
S_K^{\ell}$  and $\mu,\nu \in S_K$ we define the following polynomials in the tuple $(z_k)_{k\in S_K}$:
\[
  P_s\left((z_k)_{k \in S_K}\right) := \prod_{k=1}^\ell z_{s_k}
	\quad\text{and}\quad
	P_{\mu,\nu}\left((z_k)_{k \in S_K}\right) := \sum_{s\in C_{\mu,\nu}(\mathsf{R})} P_s\left((z_k)_{k \in S_K}\right).
\]
The system is given by
\begin{equation}
\label{eq:diffeq3}
\frac{dz_k}{dt}  = f_k\big(z_1(t), \dots, z_K(t), z_\omega(t)\big) 
\end{equation}
with initial conditions $z_1(0) = 1$ and $z_k(0) = 0$ for $k \in S_K\setminus\{1\}$, and where (omitting for brevity the argument $(z_1(t),\ldots,z_\omega(t))$)  for $k\in[K]$,
\begin{equation}
\label{eq:diffeq1}
f_k = f_k^+ - f_k^-
~~
\textrm{with}
~~
f_k^+ = k\sum_{\genfrac{}{}{0pt}{}{1\leq \mu \leq \nu}{\mu+\nu=k}}P_{\mu,\nu},
~~
f_k^- = 2kP_{k,k} +
k\sum_{\genfrac{}{}{0pt}{}{\mathclap{\mu\in S_K\setminus\{k\}}}{\phantom{\mu+\nu = k}}}
P_{\mu,k},
\end{equation}
and for $k=\omega$,
\begin{equation}
\label{eq:diffeq2}
f_{\omega} = \sum_{\genfrac{}{}{0pt}{}{1\leq \mu\leq \nu\leq K}{\mu+\nu>K}} (\mu+\nu)P_{\mu,\nu} + \sum_{\mu=1}^{K} \mu P_{\mu,\omega}.
\end{equation}
The idea behind these definitions is that if $Y_k(N)= nz_k(N/n)$ for all $k \in S_K$, then $P_s$ equals the probability that $s$ is the component size vector of $K$ randomly selected vertices (i.e., the $i$th selected vertex is in an $s_i$-component, for all $1\le i \le\ell$). Thus, $f_k^+$ is (close to) the expected number of vertices in $k$-components created in round $N+1$, and $f_k^-$ is (close to) the expected number of vertices in $k$-components destroyed in round $N+1$; this will be made precise in the proof of Lemma~\ref{lemma:linearregime}. Note that these functions depend on the underlying $(K,\ell)$-rule
$\mathsf{R}$. The following lemma justifies the specific choice of the differential equation system.
\begin{lemma} \label{lemma:linearregime}
  Let $k, \ell \in \mathbb{N}$, let $\mathsf{R}$ be a $(K,\ell)$-rule, and let $T>0$. Let $\lambda \in \omega(n^{-1}) \cap o(1)$. Then there exists a unique solution $(z_k(t))_{k\in S_K}$ of the system~\eqref{eq:diffeq3}, and with probability at least $1-O(\frac{1}{\lambda}\exp(-{n\lambda^3}/{8K^3}))$,
\[
Y_k(N) = nz_k(N/n) + O(\lambda n)
\]
uniformly for all $k\in S_K$ and all $0\leq N \leq Tn$.
\end{lemma}
In the proof of Lemma~\ref{lemma:linearregime} we use the following general statement that is a special case of~\cite[Theorem 5.1]{Wormald1999}. Assume that for every $n\geq 1$ we have a Markov chain $(G_0^{(n)},G_1^{(n)},\ldots)$, where the random variable $G_N^{(n)}$ takes values in the set $\mathcal{G}^{(n)}$ of all graphs on $n$ vertices. When referring to the Markov chain we usually drop the dependence on $n$ from the notation. In our context, $G_N = G_N^{\mathsf{R}}$. Let $\mathcal{G}^{(n)+}$ be the set of valid sequences with respect to the Markov chain, i.e. the set of all sequences $(G_0,G_1,\ldots)$ such that $G_N \in \mathcal{G}^{(n)}$, and the transition probability from $G_N$ to $G_{N+1}$ is positive for all $N\geq 0$. For functions $Y_1=Y_1^{(n)},\ldots,Y_a = Y_{a}^{(n)}\colon \mathcal{G}^{(n)} \to \mathbb{R}$, and $D\subseteq \RR^{a+1}$ we define the \emph{stopping time} $N_D(Y_1,\ldots,Y_a)$ to be the minimum $N$ such that $$(N/n,Y_1(G_N)/n,\ldots,Y_a(G_N)/n) \not\in D.$$
In our context, $a = K+1$ and $Y_{K+1} = Y_\omega$. With this notation, the following theorem holds.
\begin{theorem}[Theorem 5.1. in \cite{Wormald1999}, simplified\footnote{The theorem in~\cite{Wormald1999} is not restricted to Markov chains, and it is also not restricted to graphs. Moreover, the boundedness hypothesis may be satisfied only for a function $\beta = \beta(n)$ and may fail with some error probability $\gamma = \gamma(n)$.}]
\label{thm:diffeq}
Let $a,n \in \mathbb{N}$. For $1\leq k \leq a$ let $Y_k\colon \mathcal{G}^{(n)} \to \mathbb{R}$ and $f_k\colon\RR^{a+1} \to \RR$ be functions such that $|Y_k(G)| \le n$ for all $G \in \mathcal{G}^{(n)}$. 
Let $D$ be some bounded connected open set containing the closure of
\[
\{(0,z_1,\ldots,z_a) \mid \Pr[Y_k(G_0) = z_kn\ \text{  for all } 1 \leq k \leq a] \neq 0 \text{ for some }n \}.
\]
Assume the following three conditions hold.
\begin{enumerate}
\item[(i)] (Boundedness hypothesis) There is a constant $\beta \geq 1$ such that for all $1\leq k\leq a$, all $(G_0,G_1,\ldots) \in \mathcal{G}^{(n)+}$, and all $N\geq 0$ we have
\[
|Y_k(G_{N+1})-Y_k(G_N)| \leq \beta.
\]
\item[(ii)] (Trend hypothesis) For some function $\lambda = \lambda(n) = o(1)$ and for all $1 \le k\leq a$ and all $G \in \mathcal{G}$,
  \[
    \left|\EE[Y_k(G_{N+1})-Y_k(G_N) \mid G_N=G ] - f_k\Big(\frac{N}{n},\frac{Y_1(G)}n,\ldots,\frac{Y_a(G)}n\Big)\right| \leq \lambda
  \] 
  for all $N < N_D$.
\item[(iii)] (Lipschitz hypothesis) Each function $f_k$ is continuous, and satisfies a Lipschitz condition  on
$D \cap \{(t,z_1,\ldots,z_a) \mid t\geq 0\}$. 
\end{enumerate}
Then the following is true.
\begin{enumerate}
\item[(a)] For $(0,\hat z_1,\ldots,\hat z_a)\in D$, the system of differential equations 
\[
\frac{dz_k}{dt} = f_k(t,z_1,\ldots,z_a), \quad k=1,\ldots,a
\]
has a unique solution in $D$ for $z_k:\RR\to\RR$ passing through $z_k(0) = \hat z_k, 1\leq k\leq a$ and the solution extends to points arbitrarily close to the boundary of $D$;
\item[(b)] For some $C>0$, with probability $1-O(\frac{1}{\lambda}\exp(-{n\lambda^3}/{\beta^3}))$,
\begin{equation}\label{eq:concentration}
Y_k(G_N) = nz_k(N/n) + o(\lambda n)
\end{equation}
uniformly for $0\leq N \leq \sigma n$ and for each $k$, where $z_k(t)$ is the solution in \emph{(a)} with $\hat z_k = \frac{1}{n}Y_k(0)$, and $\sigma = \sigma(n)$ is the supremum of those $x$ to which the solution can be extended before reaching within $\ell^{\infty}$-distance $C\lambda$ of the boundary of $D$.
\end{enumerate}
\end{theorem}
\begin{proof}[Proof of Lemma~\ref{lemma:linearregime}]
\label{lem:linearregime:proof}
We apply Theorem~\ref{thm:diffeq} as follows. As domain $D$ we choose (somewhat arbitrarily) $D := (-2T,2T) \times (-1,2)^{K+1}$. Note that $D$ contains the set $[0,T]\times [0,1]^{K+1}$, as required. We will verify the conditions in Theorem~\ref{thm:diffeq} one by one.

As already mentioned, for $k\in S_K$, $Y_k(N) = Y_k(G_N^{\mathsf{R}})$ denotes the number of vertices in $k$-components in $G_N^{\mathsf{R}}$. Note that $\sum_{k\in S_K} Y_k(N) = n$ for all $N$. Then the boundedness hypothesis \emph{(i)} is met with $\beta = 2K$, since any of the $Y_k$'s, $k \in [K]$, can change by at most $2K$ when adding an edge to $G_N$. 

The functions $f_k$, $k\in S_K$ are given by~\eqref{eq:diffeq1} and~\eqref{eq:diffeq2}. To see that they satisfy the trend hypothesis \emph{(ii)}, note that the probability that $s\in S_K^\ell$ is the (truncated) component size vector of $\ell$ randomly selected vertices is $P_s((Y_k)_{k \in S_K})$. On the other hand, if $s\in C_{\mu,\nu}$, $\mu \neq \nu$ is the component size vector, then two components of size $\mu$ and $\nu$ are combined into a component of size
$\mu+\nu$, so $Y_{\mu}$ and $Y_{\nu}$ decrease by $\mu$ and $\nu$,
respectively, and $Y_{\mu+\nu}$ (or $Y_{K+1}$, if $\mu+\nu > K$) increases by $\mu+\nu$. The case $\mu=\nu$ is slightly more complicated, as it might be that both components are identical, in which case only an internal edge (or a loop) is added to the component. However, this event occurs only with probability
$O(n^{-1})$. Since all sums are over finitely many terms, the trend hypothesis is satisfied for a suitable function $\lambda \in O(n^{-1})$.

Finally, all the functions $f_k$ are polynomials in $z_1,\ldots,z_{\omega}$, so
they trivially satisfy the Lipschitz condition \emph{(iii)}. Thus, all the assumptions of Theorem~\ref{thm:diffeq} are satisfied. 

It remains to check that the solution of the differential equations does not come close to the boundary of $D$ except for the first component. Observe that $z_k(N/n)\in [0,1]$ for all $N$ for which~\eqref{eq:concentration} holds, because $0\leq Y_k(N)\leq n$. Thus, part (b) and the continuity of the $z_k$'s imply the claim.
\end{proof}
We also state a simpler but more explicit bound that will be convenient to use in the sequel.
\begin{corollary}\label{cor:linearregime}
Let $k, \ell \in \mathbb{N}$, let $\mathsf{R}$ be a $(K,\ell)$-rule, and let $T>0$. For any
$\varepsilon >0$, with probability at least $1-O(\exp(-n^{\varepsilon}))$, 
\[
Y_k(N) = nz_k(N/n) + o(n^{2/3 +\varepsilon}),
\]
uniformly for all $k\in S_K$ and all $0\leq N \leq Tn$.
\end{corollary}
\begin{proof}
  Use $\lambda := 2Kn^{-1/3 + \varepsilon}$ in Lemma~\ref{lemma:linearregime}.
\end{proof}


For later reference we first collect some basic properties of the functions $z_k$. The following lemma is in parts a generalization of Theorem 2.1 in~\cite{Spencer2008}, where the phase transition was studied in the case $\ell=4$.
\begin{lemma}
\label{lem:propdiffeq}
Let $K, \ell \in \mathbb{N}$ and let $\mathsf{R}$ be a $(K,\ell)$-rule. Then the unique solution $(z_k(t))_{k \in S_K}$ of~\eqref{eq:diffeq3} has the following properties.
\begin{enumerate}
\item[(a)] $\sum_{k\in S_K}z_k(t) =1$ for all $t \geq 0$.
\item[(b)] For all $t > 0$ and all $k\in S_K$ we have $0< z_k(t)<1$.
\item[(c)] For every $1\leq i\leq K$ the function $\sum_{k=1}^i z_k$ is strictly decreasing. Moreover, $z_{\omega}$ is strictly increasing.
\item[(d)]  \label{lem:propdiffeqd} If $\mathsf{R}$ is non-degenerate,
  then there is $t_0>0$ and $c>0$ such that $1-z_{\omega}(t) \leq
  e^{-c(t-t_0)}$ for all $t\geq t_0$. In particular, $z_{\omega}(t) \to 1$ for
  $t \to \infty$.
\end{enumerate}
\end{lemma}
\begin{proof}\emph{Proof of (a).}
In the sum $\sum_{k\in S_k} f_k(z_1,\ldots,z_{\omega})$, for $\mu
\neq \nu$ and $s\in C_{\mu,\nu}$ the term $\mu \cdot P_s$ is added and
subtracted exactly once. For $s\in C_{\mu,\mu}$ the term $\mu \cdot P_s$ is
added and subtracted exactly twice. Hence, all terms cancel, and we have
$\sum_{k\in S_k} f_k(z_1,\ldots,z_{\omega}) = 0$. Thus, the function
$\tilde{z}(t) := \sum_{k\in S_K}z_k(t)$ satisfies the differential equation
${d\tilde z}/{dt} = 0$, with initial condition $\tilde z(0) = 1$.
Therefore, $\tilde z(t) = 1$.

\vspace{5pt}
\noindent\emph{Proof of (b).}
We will show $z_k(t) >0$ for all $k\in S_K$ and all $t>0$; the other inequality follows then directly from (a).
By applying Corollary~\ref{cor:linearregime} we infer that $z_k(t) \geq 0$ for all $k\in S_K$. By (a), this implies $z_k(t) \leq 1$ for all $k\in S_K$.

First we show that if there is $t_0 > 0$ and $1 \le k\le K$ such that
$z_k(t_0)>0$, then $z_k(t)>0$ for all $t\geq t_0$. Note that
\begin{equation*}
\label{eq:lowerboundf_k}
	z_k'(t)
	\ge -f_k^-
	\ge - k \sum_{\mu\in S_K\setminus\{k\}} \sum_{s \in C_{\mu,k}} P_s - 2k\sum_{s \in C_{k,k}} P_s .
\end{equation*}
In the last expression each occurring term $P_s$ contains a factor $z_k$, and all other factors are $\leq 1$. Thus,
by abbreviating $C_k:= \sum_{\mu\in S_K\setminus\{k\}} \sum_{s \in C_{\mu,k}} 1  + 2\sum_{s \in
  C_{k,k}} 1$, we readily get that $z_k'(t) \ge-kC_kz_k(t)$. By integrating this from $t_0$ to $t$ we obtain that $z_k(t)\geq e^{-kC_k(t-t_0)} z_k(t_0)>0$ for all $t \ge t_0$.

Next, note that $z_1(0)=1>0$, so the previous argument implies that $z_1(t) >
0$ for all $t \ge 0$. We show by induction on $k$ that $z_k(t)>0$ holds for all
$t>0$ and $1 \le k \le K$. For some $2 \le k\leq K$, assume there was $t_0>0$
with $z_k(t_0) = 0$. Since $z_k(0)=0$ then again the previous argument implies
that $z_k(t)=0$ for all $0\leq t\leq t_0$, and for this
range~\eqref{eq:diffeq1} simplifies to
\[
f_k(z_1,\ldots, z_{\omega}) = \sum_{\genfrac{}{}{0pt}{}{1\leq \mu \leq \nu}{\mu+\nu=k}}\left((\mu+\nu)\sum_{s \in C_{\mu,\nu}} P_s\right).
\]
This expression is at least $\sum_{1\leq \mu \leq \nu, \mu+\nu=k}z_{\mu}^{\ell/2}z_{\nu}^{\ell/2}$, since $(\mu,\nu,\ldots,\mu,\nu) \in C_{\mu,\nu}$. The right side is positive by induction hypothesis, which contradicts the fact that $f_k(z_1,\ldots, z_{\omega}) = {dz_k}/{dt}=0$ for all $0\leq t\leq t_0$. This shows the claim for all $1 \le k \le K$ and $t > 0$.

It remains to treat the case $k=\omega$. Equation~\eqref{eq:diffeq2} implies that
\[
	f_\omega(z_1, \dots, z_\omega) \ge P_{1,K}
\]
and since we have already shown that $z_k(t) > 0$ for all $1 \le k \le K$ and $t > 0$ this expression is $> 0$ for all $t > 0$. The claim now follows with the same contradiction as for $f_k$.

\vspace{5pt}
\noindent\emph{Proof of (c).}
In the sum $\sum_{k=1}^{i} f_k(z_1,\ldots,z_{\omega})$, for all $1\leq\mu\leq
\nu\leq K$ with $\mu+\nu \leq i$ and $s\in C_{\mu,\nu}$ the term $\mu \cdot
P_s$ is added and subtracted exactly once if $\mu \neq \nu$, and it is added
and subtracted exactly twice if $\mu =\nu$. So all these terms cancel. On the
other hand, for all $1 \leq \mu\leq \nu \leq K$ with $\mu +\nu >i$, the terms
$\mu \cdot P_s$ are only subtracted (once or twice), but not added, and by (b)
all these terms are $>0$ for $t>0$. Hence, the function $\sum_{k=1}^i z_k$ has
negative derivative for all $t>0$, so it is strictly decreasing. As for the
final remark, we infer directly that $z_{\omega} = 1 -\sum_{k=1}^K z_k$ is
strictly increasing.

\vspace{5pt}
\noindent\emph{Proof of (d).} 
By definition of $C_{\mu,\nu}$ the bounded size rule is non-degenerate, if and only
if for each $1\leq k\leq K$ there exists $s\in C_{k,\omega}$ such that $P_s =
z_k\cdot z_{\omega}^{\ell-1}$. Thus, $f_{\omega}(z_1,\ldots, z_{\omega}) \geq
(\sum_{k=1}^{K} z_k)z_{\omega}^{\ell-1}$. Fix any $t_0>0$, and let $\tilde c :=
z_{\omega}(t_0)$. From (b) and (c) we know that $\tilde c>0$ and that
$z_{\omega}(t) \geq \tilde c$ for all $t\geq t_0$, respectively. Hence,
\[
\sum_{k=1}^K f_k =  - f_{\omega} \leq - \left(\sum_{k=1}^{K} z_k(t) \right)z_{\omega}(t)^{\ell-1} \leq - \left(\sum_{k=1}^{K} z_k(t) \right) \tilde c^{\ell-1}
\]
for all $t\geq t_0$. Therefore, the function $\tilde{z} := \sum_{k=1}^K z_k$
satisfies ${d\tilde z}/{dt} \leq - \tilde c^{\ell-1}\tilde z$ for all $t\geq
t_0$. Thus $0 \leq \tilde z(t) \leq \tilde z(t_0)e^{-\tilde
  c^{\ell-1}(t-t_0)}\to 0$ for $t \to \infty$. The claim now follows with
$c:=\tilde c^{\ell-1}$ from $\tilde z(t_0)\leq 1$ and $1-z_{\omega}(t) = \tilde z(t)$.
\end{proof}
We continue with a crucial ingredient for studying the fine properties of the distribution of $\Tcon^\mathsf{R}$. We determine the limiting behavior of the fraction $z_k(t)$ of vertices in $k$-components in $G_{tn}^{\mathsf{R}}$; in particular, for all $k\in \slow(\mathsf{R})$ the next lemma asserts that $z_k(t)$ approaches $C_k e^{-\ext(\mathsf{R})t}$, for some $C_k = C_k(\mathsf{R}) > 0$. 
\begin{lemma}\label{lem:limits}
Let $K, \ell\in \mathbb{N}$ and let $\mathsf{R}$ be a $(K,\ell)$-rule.
\begin{enumerate}
  \item[(a)] For every $\eps>0$ there exists a $t_0>0$ such that for all $t\ge t_0$
   $$\sum_{k\in \fast(\mathsf{R})}z_k(t) \le \eps\cdot \sum_{k\in \slow(\mathsf{R})}z_k(t).$$ 
  \item[(b)] If $\mathsf{R}$ is non-degenerate then for $k\in \slow(\mathsf{R})$ the
    limit $$c_k := \lim_{t\to\infty}(\ext(\mathsf{R}) \cdot t+\log z_k(t))$$ exists\footnote{As will be proven later, the constant  $d_k$ from Theorem~\ref{thm:main} is $d_k = 1/k \cdot e^{c_k}$}.
\end{enumerate}
\end{lemma}
\begin{proof}
For brevity we write $\slow = \slow(\mathsf{R})$, $\fast =
\fast(\mathsf{R})$ and $\ext = \ext(\mathsf{R})$. Furthermore, we let $z_{\slow}:=\sum_{k\in \slow}z_k$ and $z_{\fast}:=\sum_{k\in \fast}z_k$. Note that $z_{\slow}(t)+z_{\fast}(t)+z_{\omega}(t) \equiv 1$ and that $z_\omega(t)$ is increasing, while $z_{\slow}(t)+z_{\fast}(t)$ is decreasing.

Recall that $z_k' = f_k = f_k^+ - f_k^-$ for all $1 \leq k \leq K$, see~\eqref{eq:diffeq1}. Here $f_k^+$ consists of terms $P_s$, $s\in
C_{\mu,\nu}$, with $1\leq \mu\leq\nu\leq K$. Every such term contains at least
two factors $z_i, z_j$ with $1\leq i,j \leq K$. Since by
Lemma~\ref{lem:propdiffeq} we know that $z_i \le 1$ for all $ i \in S_K$ there
is a $c>0$ such that
\begin{equation}
\label{eq:growthab1}
0 \leq f_k^+ \leq c(z_{\slow}+z_{\fast})^2.
\end{equation}
The term $f_k^-$ sums up terms $P_s$ for indices $s \in S_K^\ell$ for which at
least one component equals $k$. Moreover, if $k\in \slow$ then the coefficient
of the polynomial $z_kz_{\omega}^{\ell-1}$ in $f_k^-$ is exactly $\ext$,
and for $k\in \fast$ it is $\geq (\ext+1)$. All other terms in $f_k^-$ contain
at least the factor $z_k$ and another factor $z_i$, $1\leq i \leq K$. Hence, by
making the constant $c>0$ from~\eqref{eq:growthab1} larger if necessary we obtain 
\begin{equation}
\label{eq:growthab2}
 \ext\cdot z_{\omega}^{\ell-1} \cdot z_k\leq f_k^- \leq (\ext + c(z_{\slow}+z_{\fast}))\cdot z_k\qquad\text{for all $k\in \slow$,}
\end{equation}
and
\begin{equation}
\label{eq:growthab3}
 (\ext+1)z_{\omega}^{\ell-1}\cdot z_k\leq f_k^- 
 \qquad\text{for all $k\in \fast$}.
\end{equation}
Consider an arbitrary $\eps >0$. By Lemma~\ref{lem:propdiffeq} (d) there
exists $t_0'>0$ such that $z_{\omega}(t_0') \ge 1-\eps^2$, and by the monotonicity of $z_\omega(t)$,
\begin{equation}
\label{eq:growthabeps}
z_{\omega}(t)\ge 1-\eps^2 \quad\text{and}\quad z_{\slow}(t)+z_{\fast}(t)\le \eps^2\qquad\text{ for all $t\ge t_0'$.}
\end{equation}
Together with $z_k' = f_k^+-f_k^-$, the lower bound in \eqref{eq:growthab1} and the upper bound in 
\eqref{eq:growthab2} imply $z_{k}'(t)\ge -(\ext+\eps) z_k(t)$ for all $0<\eps<1/c$ and $k\in \slow$. Dividing both sides by $z_k(t)$ and integrating from $t'$ to $t$ yields
\begin{equation}
\label{eq:growthabzk}
 z_{k}(t)\ge z_k(t')\cdot e^{-(\ext+\eps)(t-t')}\quad \text{ for all $k\in \slow$ and all $t\ge t' \ge t_0'$.}
\end{equation}
With the above preparations we are ready to prove the lemma. In order to see $(a)$ we first prove an auxilliary statement. We claim that whenever there exist $t_2 > t_1 \ge t_0'$ with $t_1 \in \mathbb{R}, t_2 \in \mathbb{R} \cup \{\infty\}$ such that for all $t \in [t_1, t_2)$ we have $z_{\fast}(t)\ge \eps z_\slow(t)/2$, then 
\begin{equation}
\label{eq:fracfastslow}
\frac{z_\fast(t)}{z_\slow(t)} \le \frac{z_\fast(t_1)}{z_\slow(t_1)} \cdot
e^{-(\frac{1}2 -\eps)(t-t_1)} \quad\text{for all $t \in [t_1, t_2)$.}
\end{equation}
To prove \eqref{eq:fracfastslow}, note that the assumption on $t_1$ and $t_2$, together with \eqref{eq:growthab1}, \eqref{eq:growthab3} and \eqref{eq:growthabeps}, imply that for all $t \in [t_1, t_2)$
\begin{equation*}
\begin{split}
	z_{\fast }'(t)
	& \le
	\sum_{k \in \fast} \left(c(z_\slow + z_\fast)^2 - (\ext+1)z_\omega^{\ell-1} \cdot z_k \right) \\
	& \le
	c|\fast|\eps^2(1+\tfrac2\eps)z_{\fast}(t) -
(\ext+1)(1-\eps^2)^{\ell-1} z_{\fast}(t).
\end{split}
\end{equation*}
For $\eps>0$ sufficiently small we thus have $z_{\fast}'(t)\le -(\ext+\tfrac12) z_\fast(t)$ and so
$$
z_{\fast}(t)\le z_\fast(t_1)\cdot
e^{-(\ext+\frac12)(t-t_1)}\quad\text{for all $t \in [t_1, t_2)$.}
$$
Together with \eqref{eq:growthabzk} (where we use $t' = t_1$), this implies \eqref{eq:fracfastslow}, as claimed.

Equation \eqref{eq:fracfastslow} allows us to infer $(a)$ by contradiction as follows. First of all, note that if for all $t
\ge t_0'$ we had $z_{\fast}(t)\ge \eps z_\slow(t)/2$ then we could
apply~\eqref{eq:fracfastslow} with $t_0'$ in place of $t_1$ and $\infty$ in
place of $t_2$. Since by Lemma~\ref{lem:propdiffeq} $z_\slow(t_0') > 0$ we
infer that there is a $t_0'' \ge t_0'$ such that $z_\fast(t_0'') < \eps
z_\slow(t_0'') /2$, a contradiction. So there is a $t_1' \ge t_0'$ such that $z_\fast(t_1') < \eps z_\slow(t_1') /2$. Assume for the sake of contradiction that there is a $t_2' > t_1'$ such that $z_\fast(t_2') > \eps
z_\slow(t_2')$. Then by continuity of the $z_k$'s, there would be an interval $I = [t_1'',t_2''] \subseteq [t_1',t_2']$ such that $z_\fast(t_1'') = \eps
z_\slow(t_1'')/2$, $z_\fast(t_2'') = \eps z_\slow(t_2'')$, and 
$z_\fast(t) \geq \eps z_\slow(t)/2$ for all $t\in I$. However, this is a contradiction since \eqref{eq:fracfastslow} implies that the ratio ${z_\fast(t)}/{z_\slow(t)}$ cannot increase in $I$. Thus $z_\fast(t) \le \eps
  z_\slow(t)$ for all $t \ge t_1'$; this establishes $(a)$ with $t_0 = t_1'$.

In order to prove $(b)$, by applying \eqref{eq:growthab1},
\eqref{eq:growthab2} and \eqref{eq:growthabeps} we infer that for $0 < \eps < \min\{1/c,1\}$ and $t \ge t_0$
$$
	z_{\slow }'(t)
	= \sum_{k \in \slow} (f_k^+ - f_k^-)
	\le cK\eps^2(z_{\slow}(t) + z_\fast) - \ext \cdot (1-\eps^2)^{\ell-1} z_{\slow}(t).
$$
By using $(a)$ we further get for $t \ge t_0$
$$
	z_{\slow }'(t)
	\le 2cK\eps^2z_{\slow}(t)  - \ext \cdot (1-\eps^2)^{\ell-1} z_{\slow}(t).
$$
For all $\eps > 0$ small enough we thus get $z_{\slow}'(t)\le -(\ext-\eps) z_\slow(t)$ and so
\begin{equation}
\label{eq:growthabzslow}
z_{\slow}(t)\le z_\slow(t_0)\cdot e^{-(\ext-\eps)(t-t_0)}\quad\text{ for all
$t\ge t_0$.}
\end{equation}
Since $\mathsf{R}$ is non-degenerate we have $\ext \ge 1$.  Together with
\eqref{eq:growthabzk} (applied to $t' = t_0$) this implies that for $\eps>0$ small enough there exists
a constant $C>0$ such that 
\begin{equation}
\label{eq:growthabratio}
\frac{(z_{\slow}(t))^2}{z_k(t)} \le C e^{-t/2}\quad\text{ for all $k\in \slow$ and all $t\ge t_0$.}
\end{equation}
Next we use~\eqref{eq:growthab2} and (a) again to obtain for $k\in \slow$ and $t \ge t_0$
$$
	z_k'(t) \ge -f_k^-
	\ge - (\ext + c(z_\slow(t) + z_\fast(t)))z_k(t)
	\ge - (\ext + 2cz_\slow(t))z_k(t)
$$
and similarly, using \eqref{eq:growthab1},~\eqref{eq:growthab2}, and $z_\omega(t) = 1 - z_\fast(t) - z_\slow(t) \ge 1 - 2z_\slow(t)$
$$
	z_k'(t)
	= f_k^+ - f_k^-
	\le 4c\cdot (z_\slow(t))^2 - \ext\cdot (1 - 2z_\slow(t))^{\ell-1} z_k(t).
$$
Thus, for $t \ge t_0$ 
$$
-(\ext + 2cz_{\slow}(t))  \le \frac{z_k'(t)}{z_k(t)} \le  \frac{4c\cdot (z_\slow(t))^2}{z_k(t)} - \ext\cdot (1 - 2\ell z_\slow(t)).
$$
Now we integrate all three sides from $t_1$ to $t_2$. By~\eqref{eq:growthabzslow} and~\eqref{eq:growthabratio}, all terms on the left and the right hand side decay exponentially, except for the constant $-\ext$. Therefore, there exists $\hat t_0 \ge t_0$ so that for $t_2 > t_1\ge \hat t_0$
$$
-\ext\cdot (t_2-t_1) - \eps   \le \log z_k(t_2) - \log z_k(t_1) \le - \ext\cdot
(t_2-t_1)  + \eps.
$$  
That is, the sequence $\ext \cdot t+\log z_k(t)$ is a Cauchy sequence and thus convergent.
\end{proof}
We close this section with a statement about the existence of a large component in random graph processes, which is also true if we begin with a graph that already contains some edges. 
\begin{lemma}\label{lemma:universalfastgiant}
Let $\varepsilon >0$ and let $\ell>0$ be an even integer. Then there exists a constant $C_\varepsilon > 0$ such that for any $\ell$-Achlioptas process with arbitrary initial graph $G_0$ on $n$ vertices, after at most $C_\varepsilon n$ rounds the number of vertices in the largest component is at least $(1-\varepsilon)n$ with probability $1-o(1/n)$.
\end{lemma}
\begin{proof}
We may assume that $\eps < 1/3$. Then observe that as long as there is no component with $\ge (1-\varepsilon)n$ vertices there exist two disjoint vertex sets $A$ and $B$ (not necessarily the same in each round) such that $|A|,|B|\geq\varepsilon n$ and such that no component contains vertices both from $A$ and $B$. Note that this assumption on $A$ and $B$   implies that every edge between $A$ and $B$ connects two different components. Observe also that the probability that the $\ell$-Achlioptas process will choose such an edge is at least $p_\varepsilon:=\eps^{\ell}$, as this is at least the probability that $(v_1,\ldots,v_{\ell})\in (A\times B)^{\ell/2}$.
    
Set $C_\varepsilon := 2/p_{\varepsilon}$ and assume that for $C_\varepsilon n$ rounds the size of the largest component is $\le (1-\eps)n$. From the previous discussion we know that this occurs with probability at most $\Pr[\text{Bin}(C_\varepsilon n, p_{\varepsilon}) < n]$, which is easily seen to be $o(1/n)$ by the choice of $C_\varepsilon$ and the Chernoff bounds.
\end{proof}


\section{Late Stages of the \texorpdfstring{$\mathsf{R}$}{R}--process and Proof of Theorem~\ref{thm:main}}\label{sec:late}

As described in the previous chapter, the differential equation method allows us to analyze $(K,\ell)$-rules in (relatively) early stages of the process, i.e., when a linear number of edges is added to an initially empty graph. However, as we argue in the sequel, the graph will become connected much later, after roughly $n\log n/\ext(\mathsf{R})$ rounds. For this number of rounds the differential equation method fails. Although our forthcoming result will imply that the concentration result of Lemma~\ref{lemma:linearregime} can essentially be extended up to the point where $(1/\ext(\mathsf{R})-o(1))n\log n$ edges have been added, it is not clear how such a statement can be established via a method based on differential equations. We will instead take an alternative and more direct route, where we prove inductively that the post-linear regime follows typically a suitably defined deterministic trajectory. In the remainder of this section we write $a= x \pm y$, where $y > 0$, for $a\in [x-y, x+y]$.
\additions{
\begin{lemma}
Let $K,\ell \in \mathbb{N}$ and let $\mathsf{R}$ be a non-degenerate $(K,\ell)$-rule.
For $1 \le k \le K$, let $Y_k(N)$ denote the number of vertices in components
with $k$ vertices in $G_{N}^{\mathsf{R}}$. Moreover, for $k \in
\slow(\mathsf{R})$ let $c_k = c_k(\mathsf{R})$ be defined as in
Lemma~\ref{lem:limits}.  For $t \in \mathbb{R}_+$ let $N_t =
tn/{\ext(\mathsf{R})}$.

Then there is an $\eps_0 > 0$ such that for any $0 < \eps < \eps_0$ there is a $t_0 = t_0(\eps)$ such that for all $t_0 \le t \le (1-\eps) \log n$ whp
\begin{enumerate}[(a)]
\item For all $k \in \slow(\mathsf{R})$, $Y_k(N_t) = (1 \pm \eps) \cdot e^{c_k - t} \, n$.
\item $\sum_{k\in \fast(\mathsf{R})}Y_k(N_t) \le \eps \, e^{-t} \,
  n$.\marginpar{\tiny\he{I've changed part (b), need to look over the proof to
  make sure it's tackled or modify it. Note that this statement should apply to
  $K = (1-\varepsilon)n$ which needs to change.}}
\end{enumerate}
\end{lemma}
}
\begin{lemma}
\label{lemma:phase2}
Let $K,\ell \in \mathbb{N}$ and let $\mathsf{R}$ be a non-degenerate $(K,\ell)$-rule.
For $1 \le k \le K$, let $Y_k(N)$ denote the number of vertices in components
with $k$ vertices in $G_{N}^{\mathsf{R}}$. Moreover, for $k \in
\slow(\mathsf{R})$ let $c_k(\mathsf{R})$ be defined as in
Lemma~\ref{lem:limits}.  For $t \in \mathbb{R}_+$ let $N_t =
tn/{\ext(\mathsf{R})}$. Then there is an $\eps_0 > 0$ such that for any $0 < \eps < \eps_0$ there is a $t_0 = t_0(\eps)$ such that with probability $1-o(1/\log n)$ for all $t_0 \le t \le (1-\eps) \log n$:
\begin{enumerate}
\item[(a)] $Y_k(N_t) = (1 \pm \eps) \cdot e^{c_k(\mathsf{R}) - t} \, n$ for all $k \in \slow(\mathsf{R})$,  and
\item[(b)] $Y_k(N_t) \le \eps \, e^{-t} \, n$  for all $k\in \fast(\mathsf{R})$.
\end{enumerate}
\end{lemma}

\begin{proof}
Recall that an event holds \emph{with log-high probability} (wlhp) if it
holds with probability $1-o(1/\log n)$ for $n\to \infty$. We use an inductive argument, where Corollary~\ref{cor:linearregime} provides us with the base case. Indeed, for any fixed $T>0$, by Corollary~\ref{cor:linearregime} there exists $\tilde\delta > 0$ such that wlhp
\begin{equation}\label{eq:phase2start}
Y_k(N_t) = nz_k(t/\ext(\mathsf{R})) + o(n^{1-\tilde\delta}) \quad \text{for all $1 \le k \le K$ and $t \leq T$}.
\end{equation}
Recall that by Lemma~\ref{lem:limits} for $k\in \slow(\mathsf{R})$ the limit $\lim_{t\to\infty}\left(\ext(\mathsf{R}) \cdot t + \log z_k(t)\right)$ exists and equals $c_k:=c_k(\mathsf{R})$. Thus we may choose $t' \in \mathbb{R}^+$ so that for all $t \ge t'$
\[
	|t + \log z_k(t/\ext(\mathsf{R})) - c_k| \le \eps^2/4.
\]
By rearranging, we obtain that $z_k(t/\ext(\mathsf{R})) = (1 \pm \eps^2/2)
\cdot e^{c_k} e^{-t}$ for sufficiently small $\eps$. Choose $t'' \in \mathbb{R}^+$ such that for all $t \ge
t''$ we have $e^{c_k-t}\le \eps^2/4$ for all $k \in
\slow(\mathsf{R})$ . Then in particular we have for all $t \ge t''$ and all $k \in
\slow(\mathsf{R})$ that $e^{-c_k} \, z_k(t/\ext(\mathsf{R})) \le \eps^2/2$.
Finally, from~Lemma~\ref{lem:limits} it follows that we can choose a $t''' \in \mathbb{R}^+$
such that for all $t \ge t'''$
\begin{equation}
\label{eq:fasttiny}
	\sum_{k \in \fast(\mathsf{R})} z_k(t) \le \eps^2/(2K \max_{k \in \slow(\mathsf{R})}e^{c_k}) \cdot \sum_{k \in \slow(\mathsf{R})} z_k(t).
\end{equation}
Set $\tilde t_0 = \max\{t', t'', t'''\}$. Then, by definition of $\tilde t_0$ and~\eqref{eq:phase2start} with $T = \tilde t_0+1$ we have with probability $1-o(1/n)$
\[
	Y_k(N_{t}) = (1 \pm \eps^2) \cdot e^{c_k-t} n
\quad \text{for all } k\in\slow(\mathsf{R}) \text{ and all } \tilde{t}_0 \leq t \leq \tilde{t}_0+1.
\]
This shows with room to spare (a) for $\tilde t_0\leq t \leq
\tilde{t}_0+1$.  Moreover, from~\eqref{eq:phase2start}
and~\eqref{eq:fasttiny} we immediately obtain that with probability
$1-o(1/n)$ we have $Y_k(N_{t}) \le
\eps^2 e^{-t} n$, for all $k \in \fast({\mathsf{R}})$ and all $\tilde{t}_0 \leq
t \leq \tilde{t}_0+1$; this shows with room to spare (b) for all $\tilde{t}_0
\leq t \leq \tilde{t}_0+1$. These statements will serve as the base case of our
induction, so fix any $t_0 \in [\tilde{t}_0, \tilde{t}_0+1]$.

The reason why we showed (a) and (b) with much smaller error terms than required (for $t = t_0$) is that in order to cover all cases $t_0 \le t \le (1-\eps)\log n$ we will prove inductively a statement in which the error terms will gradually increase. Formally, let
\begin{equation}
\label{eq:fipm}
        \eps_0 := \eps^2
	\quad \text{and} \quad
	\eps_{i+1} = \eps_i + c'\left(e^{-t_i} + e^{t_i/2}\sqrt{\log n/n}\right) \text{~for~} i \in \N_0
\end{equation}
for a constant $c'>0$ that we will fix later (and that will not depend on $\eps$). 
By expanding the recursive definition it is easy to see that for $n$ sufficiently large we have  $\eps_i \le \eps^2 +  2c'e^{-t_0} \le c'' \eps^2$, by choice of $t_0$. Note that $c''$ does not depend on $\eps$, as $c'$ does not. If we thus choose $\eps < 1/c''$ we obtain that  $\eps_i \le \eps$ for all $i \in \N_0$.

Recall that our choice of $\tilde t_0$ implies that for all $t_0 \in [\tilde t_0, \tilde t_0+1]$ we have $Y_k(N_{t_0}) = e^{c_k-t_0} n (1\pm \eps_0)$. Let $t_i := t_0 + i$ for $i\in \N$ and write
$N_i = N_{t_{i}}$ to simplify notation. With this notation at hand it suffices to prove that with probability $1-o((n\log n)^{-1})$ we have for all $1\le i\le (1-\eps)\log n$ that
\begin{equation}
\label{eq:inductionSlow}
	Y_k(N_i) = e^{c_k - t_i}n \cdot (1\pm \eps_i) \qquad \text{ for  } k \in \slow(\mathsf{R})
\end{equation}
and
\begin{equation}
\label{eq:inductionFast}
	Y_k(N_i) \le \eps_i e^{-t_i}\, n \text{ for  } k\qquad \in \fast(\mathsf{R}).
\end{equation}
Then the proof of the lemma is completed by a union bound over the choice of $t_0$. 

We use induction over $i$. We already know that the claim is true for $i=0$.
For the induction step we will show that assuming the claim holds for some
$i\ge 0$ it also holds for $i+1$ with probability $1-o(n^{-2})$.

A round $N_i \le N < N_{i+1}$ is called \emph{$N_i$-regular} if out of the $\ell$ randomly selected vertices there is at most one vertex $v$ that is contained in a small component (i.e., in a $k$-component for $1 \le k \le K$) in $G_{N_i}^\mathsf{R}$, and $N_i$-non-regular otherwise. Note that the definition refers to the graph $G_{N_i}^\mathsf{R}$, not to $G_N^\mathsf{R}$. This seemingly strange definition has the advantage that for different rounds the events that a specific round is $N_i$-regular are \emph{independent}. In particular, let $I$ be the number of $N_i$-non-regular rounds. The induction assumption guarantees that the total number of vertices in $k$-components, where $1 \le k \le K$, is in $O(e^{-t_i}n)$. Thus, the probability that any succeeding round is $N_i$-non-regular is in $O(e^{-2t_i})$. Since $N_{i+1} - N_i = n/\ext(\mathsf{R})$ it follows that the expected number of non-regular rounds is at most $c_1 e^{-2t_i}n$, for some $c_1 > 0$. By the Chernoff bounds there is a constant $C>0$ (not depending on $\eps$) such that
\begin{equation}
\label{eq:tailboundsI}
	\Pr\left[I \ge \max\left\{Ce^{-2t_i}n, e^{-t_i/2}\sqrt{n}\right\}\right] \leq \min\left\{2^{-Ce^{-2t_i}n},2^{-e^{-t_i/2}\sqrt{n})}\right\} =o(n^{-2}) 
\end{equation}
as $t_i \le t_0+(1-\eps)\log n$.

In order to prove the induction step consider some $1\le k \le K$. Denote a position $1 \le p \le \ell$ as \emph{$k$-good} if a vector $(\omega,\ldots,\omega,k,\omega,\ldots,\omega)$ with the $k$ at position $p$ results in merging the $k$-component with an $\omega$-component. Recall that the number of $k$-good positions is exactly $\ext_k(\mathsf{R})/k$.

For ease of notation we use $X_k(N)$ to denote the number of $k$-components in $G_{N}^{\mathsf{R}}$. Clearly, $X_k(N) \equiv Y_k(N) / k$ for all $k\in [K]$. To give bounds on $X_{k}(N_{i+1})$, let $Z$ be the number of $k$-components $C\in \comp_k(G_{N_i}^{\mathsf{R}})$ that never appear in a $k$-good position during rounds $[N_i,N_{i+1})$. Observe that if a $k$-component appears in an $N_i$-regular round $N \geq N_i$, then it is merged with an $\omega$-component if and only if its position is $k$-good (since the other vertices in this round belong to $\omega$-components in $N_i$ and in all subsequent rounds). Thus $Z$ counts basically the number of $k$-components in round $N_{i+1}$, miscounting only components that appear in $N_i$-non-regular rounds. Since there are at most $\ell I$ such components,
\begin{equation}\label{eql11:zk}
Z-\ell I \le X_k(N_{i+1}) \le Z + \ell I.
\end{equation}
In the sequel we bound $Z$. We proceed as follows: enumerate the $X_k(N_i)$ $k$-components in $G_{N_i}^{\mathsf{R}}$ from $1$ to $X_k(N_i)$ in an arbitrary but fixed way and let $Z_{j,s}$ be a Bernoulli random variable that is one if and only if in round $j$ we choose a vertex from the $s$th component in a $k$-good position. Thus,
\[
\Pr[Z_{j,s}=1] = 1- \left(1- \frac{k}{n}\right)^{\ext_k(\mathsf{R})/k} = \frac{\ext_k(\mathsf{R})}{n}\pm \frac{c}{n^2}
\]
for a constant $c$ that depends on $K$ and $\ell$ but not on $\eps$. Clearly,
$$Z = \sum_{s=1}^{X_k(N_i)} Z_s,~~\text{where $Z_s$ is the indicator function for $\sum_{j=N_i+1}^{N_{i+1}} Z_{j,s} = 0$}.$$
Therefore, using the induction assumption on $X_k(N_i)$ and bounding the error terms very generously, we get
\begin{align}\label{eql11:z}
\EE[Z]
&= X_k(N_i) \cdot \left(1- \frac{\ext_k(\mathsf{R})}{n}\pm \frac{c}{n^2} \right)^{{n}/{\ext(\mathsf{R})}} \\
& = e^{-\ext_k(\mathsf{R})/\ext(\mathsf{R})}\cdot X_k(N_i)  \cdot \left(1\pm 2c/n\right)\nonumber\\
&= e^{-\ext_k(\mathsf{R})/\ext(\mathsf{R})}\cdot X_k(N_i)   \pm \tilde c e^{-2t_i}n, \nonumber
\end{align}
for some constant $\tilde c$ that  depends on $K$, $\ell$ and $\max_{k\in\slow(\mathsf{R})} c_k(\mathsf{R})$ but not on $\eps$. 

Note that for $k\in\slow(\mathsf{R})$ (i.e., for $\ext_k(\mathsf{R}) = \ext(\mathsf{R})$) the expectation agrees with the prediction of the statement: the main term is $X_k(N_i)/e$, as desired. We thus just need to show that $Z$ is concentrated. For the alert reader this should come as no surprise: the variables $X_{j,s}$ are defined similarly as in a balls-and-bin game where  the variable $Z$ counts the number of empty bins. It is well known that in a balls-and-bin game the variables are negatively associated and one can thus apply Chernoff bounds to the variable $Z$. 
Adapted to our scenario we can argue as follows.    For a fixed round $j$, the random variables $(Z_{j,s})_s$ are Bernoulli random variables that sum up to at most $1$. By adding an additional variable $Z_{j,0} := 1-\sum_{s\ge 1}Z_{j,s}$ we may thus assume that they sum up to exactly one and~\cite[Lemma 8]{Dubhashi96} thus implies that these variables are negatively associated. Moreover, for $j\not=j'$ the variables are independent and~\cite[Lemma 7]{Dubhashi96} thus implies that the whole sequence $(Z_{j,s})_{j,s} $ is also negatively associated. Finally, the functions $Z_s$ are given by applying a decreasing function to the variables $(Z_{j,s})_j$ and~\cite[Lemma 7]{Dubhashi96} thus implies that the variables $Z_s$ are also negatively associated. Therefore, we may apply the Chernoff bound to $Z = \sum_{s=1}^{ X_k(N_i)}Z_s$~\cite[Prop.\! 5]{Dubhashi96}.
Using \eqref{eql11:z} and our induction assumption on $X_k(N_i)$ we thus obtain
\begin{equation}
\label{eq:concentrationofXk}
	\Pr\left[|Z - \EE[Z]| \ge C e^{-t_i/2} \sqrt{n\log n}\right] \leq e^{-C\log n/3} = o(n^{-2}),
\end{equation}
for an appropriately chosen constant $C>0$ (not depending on $\eps$).

It  remains to collect the pieces. From \eqref{eql11:zk},~\eqref{eql11:z} and~\eqref{eq:concentrationofXk} we obtain that with probability $1-o(n^{-2})$,
\begin{align*}
X_k(N_{i+1}) & = e^{-{\ext_k(\mathsf{R})}/{\ext(\mathsf{R})}}\cdot  X_k(N_{i}) \pm \tilde c e^{-2t_i}n \pm Ce^{-t_i/2} \sqrt{n\log n} \pm \ell I.
\end{align*}
We can bound the effect of $I$ by using~\eqref{eq:tailboundsI}. We immediately observe that bounds terms in~\eqref{eq:tailboundsI} are of the same form (or smaller) than the terms that we already have. We can thus incorporate the effect of $I$ by just increasing the constants in the error terms.
For  $k\in\slow$ we thus get from the induction assumption that 
$$
X_k(N_{i+1}) = \frac1k \, e^{c_k-t_{i+1}}n \cdot \left(1\pm \eps_i \pm c' e^{-t_i} \pm c' e^{t_i/2} {\textstyle\sqrt{\log n/n}}\right),
$$
for an appropriate constant $c'>0$ that does not depend on $\eps$. This proves the inductive step for $k\in \slow(\mathsf{R})$, cf.\ the definition of $\eps_{i+1}$. The claim for $k\in\fast$ follows similarly.
\end{proof}
We are now ready to prove the main theorem, which we restate here in a slightly stronger form. 
\begin{reptheorem}{thm:main}
Let $K,\ell \in \mathbb{N}$ and let $\mathsf{R}$ be a non-degenerate $(K,\ell)$-rule. For $1 \le k \le K$ let $Y_k(N)$ denote the number of vertices in $k$-components in $G_{N}^{\mathsf{R}}$. Moreover, for $k \in \slow(\mathsf{R})$ let $d_k := e^{c_k(\mathsf{R})}/k$, where $c_k(\mathsf{R})$ is defined as in Lemma~\ref{lem:limits}. Then, if $\ext(\mathsf{R})<2K+2$ the following statements are true.
\begin{enumerate}
\item[(a)] For any $c \in \mathbb{R}$, with probability\footnote{This is stronger than the statement given in the introduction, and it is needed in the proof of part (c).} $1-o(1/\log n)$ we have for all $N \geq (n\log n+cn)/\ext(\mathsf{R})$ and all $k\in \fast(\mathsf{R})$ that $Y_k(N) = 0$, and there is only one component with more than $K$ vertices in $G_N^{\mathsf{R}}$. 
\item[(b)] For any $c \in \mathbb{R}$,
\[
\lim_{n\to\infty}\Pr\left[\Tcon^{\mathsf{R}} \leq \frac{n\log n + cn}{\ext(\mathsf{R})}\right] =  \prod_{k\in \slow(\mathsf{R})}e^{-d_ke^{-c}}.
\]
\item[(c)] Let $c_0 := \log\left(\sum_{k\in \slow(\mathsf{R})}d_k\right)$. Then
\[
\EE[\Tcon^{\mathsf{R}}] = \frac{n\log n+\gamma n +c_0n}{\ext(\mathsf{R})} + o(n).
\]
\item[(d)] For $k\in [K]$, let $T_k^\mathsf{R} := \min\{T \mid \forall N \geq T: Y_k(N) = 0\}$
  be the time at which the last $k$-component vanishes. Then $\Pr[T_k^{\mathsf R}=\Tcon^{\mathsf R}]
  \stackrel{n\to\infty}{\longrightarrow} 0$ for $k \in \fast(\mathsf{R})$, and
  for $k \in \slow(\mathsf{R})$,
\[
  \Pr[T_k^{\mathsf R}=\Tcon^{\mathsf R}] \stackrel{n\to\infty}{\longrightarrow} \frac{d_k}{\sum_{i\in\slow(\mathsf{R})}d_i}.
\]
\end{enumerate}
\end{reptheorem}
\begin{proof}
Recall that an event holds \emph{with log-high probability} (wlhp) if it
holds with probability $1-o(1/\log n)$ for $n\to \infty$. Throughout, given
$\delta>0$ and $c\in \mathbb R$, we use the following notation:
\[
  N_{\delta} := \left\lfloor \frac{(1-\delta)n\log n}{\ext(\mathsf
  R)}\right\rfloor \qquad N_c := \left\lfloor\frac{n\log n +c n}{\ext(\mathsf
  R)}\right\rfloor \qquad N_{\infty} := 2\left\lfloor \frac{n\log
  n}{\ext(\mathsf R)}\right\rfloor.
\]
To avoid any confusion, note that we use the formula for $N_c$ also for
different values of $c$. For example, we use $N_{c+\eps}$ in the obvious
meaning.

For all statements we will make use of the following basic observation. Fix
some $\delta < 1/2$. By Lemma~\ref{lemma:phase2}, the total number of vertices
in small components (in components of size $k$ for some  $1 \le k \le K$) is
wlhp in $O(n^{\delta})$, and this number cannot increase in succeeding
rounds. Similarly as in the proof of Lemma~\ref{lemma:phase2}, we call a round
\emph{regular} if at least $\ell-1$ of the randomly selected vertices belong to $\omega$-components. Let $\mathcal{E}_\delta$ be the event that all rounds between $N_\delta+1$ and $N_\infty$ are regular. The probability that a round is not regular is in $O(n^{2\delta-2})$, and so $\Pr[\mathcal{E}_\delta] = 1-O(N_\infty n^{2\delta-2}) =1-o(1/\log n)$. So $\mathcal{E}_\delta$ will occur wlhp. Note that $\mathcal{E}_\delta$ implies that no new $k$-component, where $1 \le k \le K$, is created between rounds $N_\delta$ and $N_\infty$. 

We will make frequent use of this observation in the following way. Let $1\leq k \leq K$. As in the proof of Lemma~\ref{lemma:phase2}, we denote a position $1 \le p \le \ell$ as \emph{$k$-good} if a vector $(\omega,\ldots,\omega,k,\omega,\ldots,\omega)$ with the $k$ at position $p$ results in merging the $k$-component with an $\omega$-component. Recall that the number of $k$-good positions equals $\ext_k(\mathsf{R})/k$. Now assume that $C \in \comp_k\big(G_{N_\delta}^{\mathsf{R}}\big)$ is a $k$-component in round $N_\delta$. Then the probability that in some fixed round $N \geq N_\delta$, the component $C$ does not appear in a $k$-good position is exactly $(1-k/n)^{\ext_k(\mathsf{R})/k}$.

Note that a $k$-component that appears at a $k$-good position is merged with an $\omega$-component, unless the round is not regular. Since wlhp there is no non-regular round between rounds $N_\delta$ and $N_\infty$ we have, using Markov's inequality, for any $0 < \delta < 1/2$, $M\in \mathbb{N}$ and $N \in [N_\delta, N_\infty)$
\begin{equation}
\label{eq:thm1preparation}
\begin{split}
	\Pr\Big[Y_k(N) >0  \mid & |\comp_k\big(G_{N_\delta}^{\mathsf{R}}\big)| \le M\Big] \\
	& \le M \cdot \left(1-{k}/{n}\right)^{(N-N_\delta)\ext_k(\mathsf{R})/k}+ o(1/\log n).
\end{split}
\end{equation}
With these preparations we come to the proof of the specific statements.
 
\vspace{10pt}
\noindent\emph{Proof of (a).} We first prove the statement for all $N \in [N_c, N_\infty)$. Let $k\in \fast(\mathsf{R})$. Then the right hand side of~\eqref{eq:thm1preparation}, applied for $N = N_c$ and $M= n^{\delta}$, is $o(1/\log n)$, since $\ext_k(\mathsf{R}) \ge \ext(\mathsf{R}) + 1$ and $N_c-N_\delta = n(\delta \log n+c)/\ext(\mathsf{R})$. As wlhp all rounds between $N_\delta$ and $N_\infty$ are regular, and since a small component can only be created in a non-regular round, this shows that 
\[
	\Pr\left[\forall N\in [N_c,N_\infty), k\in\fast(\mathsf{R}):~Y_k(N) = 0\right]
	= 1-o(1/\log n).
\]
Actually, we can say a little more. We interpret $\mathsf{R}$ as a $(K',\ell)$-rule $\mathsf{R}'$ for $K' := 24(K+1)$, as outlined in Remark~\ref{rem:extension} (a). Note that $K' >12\ext(\mathsf{R})$, since by assumption $\ext(\mathsf{R}) < 2K+2$. Then $\{K+1,\ldots,K'\} \subseteq \fast(\mathsf{R}')$ and $\ext(\mathsf{R}') = \ext(\mathsf{R})$, so by the same argument as before wlhp all components of these sizes (i.e.\ in $\fast(\mathsf{R}) \cup \{K+1, \dots, K'\}$) will be extinct at round $N_c$.

Next we show that in round $N_c$, wlhp all vertices that are in an $\omega$-component are actually contained in a single component. Fix $\delta = 1/3$, and note that the event $\mathcal{E}_\delta$ guarantees that no new $\omega$-component (i.e., with more than $K'$ vertices) is created between round $N_\delta$ and $N_c$. We apply the same idea as before, but now instead of regular rounds we consider $\omega$-rounds, i.e., rounds in which all $\ell$ randomly selected vertices are in $\omega$-components. Let $\mathcal{X}_\omega(N)$ be the set of components in $G_N^\mathsf{R}$ with more than $K'$ and less than $n/2$ vertices, and set $X_\omega(N) = |\mathcal{X}_\omega(N)|$. Moreover, let $\eps > 0$ be so small that $(1-\eps)^{\ell -1} > 1/2$. Then by Lemma~\ref{lemma:universalfastgiant}, wlhp at round $N_\delta$ there is a giant component with $\geq (1-\eps)n > n/2$ vertices; we call this event $\mathcal{E}_{\text{giant}}$. Then it suffices to show that wlhp for every $C\in \mathcal{X}_\omega(N_\delta)$ an edge between $C$ and the giant is inserted until round $N_c$.

Fix a component $C\in \mathcal{X}_\omega(N_\delta)$, and consider some $N \in [N_\delta,N_c]$. If all chosen vertices $v_1,\ldots,v_\ell$ of the $N$th round are in $\omega$-components, then $\mathsf{R}$ will select some edge, say $\{v_{2i-1}, v_{2i}\}$ for some $1 \le i \le \ell/2$. So, if one of $v_{2i-1}$ and $v_{2i}$ is in $C$, and the other $\ell-1$ vertices are in the giant, then $C$ will be connected to the giant. The probability that this happens is at least $2K'(1-\eps)^{\ell-1}/n>K'/n$. In a regular round, no new large component is created. Hence, by the same argument as for~\eqref{eq:thm1preparation}, and using $N_c -N_\delta = \frac{n\log n/3 +cn}{\ext(\mathsf R)} \geq \frac{n\log n}{6\ext(\mathsf R)}$ for sufficiently large $n$ and $X_\omega(N_\delta) \le n$, we may bound
\begin{align*}
	\Pr[X_\omega(N_c) >0 \mid \mathcal{E}_{\text{giant}}]
	& \leq n\,  \left(1-\frac{K'}{n}\right)^{N_c-N_{\delta}} + o({1}/{\log n})\\
	& \leq n e^{-K'\log n/6\ext(\mathsf R)} + o({1}/{\log n})
	= o({1}/{\log n}).
\end{align*}
Since all rounds between $N_c$ and $N_\infty$ are regular wlhp, this proves (a) for all rounds $N \in [N_c,N_\infty)$. To see the claim for $N \ge N_\infty$ recall that wlhp there are $O(n^{\delta})$ components with at most $K$ vertices in $G_{N_\delta}^\mathsf{R}$. Moreover, since $\mathsf{R}$ is non-degenerate we have $\ext_k(\mathsf{R}) \ge \ext(\mathsf{R}) \ge 1$ for all $1 \le k \le K$. Together with~\eqref{eq:thm1preparation}, applied for $N=N_\infty=2\lfloor n\log n /\ext(\mathsf{R})\rfloor$ and any $0 < \delta < 1/2$, this implies that
\[
	\Pr[\forall 1 \le k \le K:~Y_k(N_\infty) =0] = 1-o(1/\log n).
\]
Thus, wlhp $G_{N_\infty}^\mathsf{R}$ is connected. Hence, the claim also follows for $N \ge N_\infty$.

\vspace{10pt}
\noindent\emph{Proof of (b).} We will resort to the so-called \emph{method of
moments}. Suppose that we have $r$ sequences $Z_i(1), Z_i(2),
\dots$ of random variables, $1\leq i \leq r$, with support on $\mathbb{N}_0$.
Suppose further that there are $\lambda_1,\ldots,\lambda_k > 0$ such that for
all $e_1,\ldots,e_r \in \mathbb{N}_0$ $$\EE[Z_1^{\underline{e_1}}(N)\cdot
Z_2^{\underline{e_2}}(N) \cdots Z_r^{\underline{e_r}}(N)]  \to
\lambda_1^{e_1}\cdots \lambda_r^{e_r} \qquad \text{as }
N\to\infty,$$
where $n^{\underline{x}} := n(n-1) \cdots (n-x+1)$. Then the joint
distribution of the $Z_k(N)$ converges to the joint distribution of
independent Poisson random variables with parameters
$\lambda_1,\ldots,\lambda_r$, i.e.\ for all $z_1,\ldots,z_r \in \mathbb{N}_0$
we have $\Pr[Z_1(N) = z_1 \wedge \ldots \wedge Z_r(N)=z_r] \to \prod_{1\leq
k\leq r}e^{-\lambda_k}\lambda_k^{z_k}/z_k!\,$ as $N \to \infty$, see
e.g.~\cite[Theorem 6.10]{RandomGraphsBook}.

Let $\delta>0$ be sufficiently small and recall that $N_\delta =
\lfloor {(1-\delta)n\log n}/{\ext(\mathsf{R})}\rfloor$. For any $\eps = \eps(n)>0$, let $\mathcal{E}(\eps)$ be the event that for $k\in\slow(\mathsf{R})$ we have $|\comp_k(G_{N_\delta}^{\mathsf R})| = (1\pm \eps)d_kn^\delta$. By Lemma~\ref{lemma:phase2}, $\Pr[\mathcal{E}(\eps_0)] =1-o(1)$ for every $\eps_0>0$. By a standard argument there exists also a (possibly very slowly converging) sequence $\eps = \eps(n) = o(1)$ such that $\Pr[\mathcal{E}(\eps)] =1-o(1)$ for $n\to \infty$.

For every $k\in\slow(\mathsf R)$, every $N\geq N_\delta$, and every $C\in \comp_k(G_{N_\delta}^{\mathsf R})$, let $Z(C)$ be a Bernoulli random variable that is $1$ if $C$ does not appear in a $k$-good position between rounds $N_\delta$ and $N$ (recall that a position is called \emph{$k$-good} if a $k$-component that appears in this position in a regular round is merged into an $\omega$-component). Moreover, for every $k\in\slow(\mathsf R)$ and $N\geq N_\delta$ let $Z_k(N) := \sum_{C\in \comp_k(G_{N_\delta}^{\mathsf R})} Z(C)$. We will apply the method of
moments to the random variables $Z_k(N_c)$ in the conditional space in which ${\cal E}(\eps)$ occurs.
More precisely, we will show that for every vector $e\in \mathbb N_0^{K}$ we have
for all large enough $n$ that
\begin{equation} \label{eq:fallingmoments}
  \EE\left[\prod_{k\in\slow(\mathsf{R})}Z_k^{\underline{e_k}}(N_c) \mid \mathcal E(\eps)\right]
	= \prod_{k\in\slow(\mathsf{R})} ((1\pm3\eps)d_ke^{-c})^{e_k}.
\end{equation}
This implies the claim as follows: recall that $\mathcal E_\delta$ is the event that all rounds between $N_\delta$ and $N_\infty$ are regular. Then by (a) and since $\Pr[\mathcal E_\delta]=1-o(1)$ we have
\[
\begin{split}
  \Pr[\Tcon\leq N_c \mid \mathcal{E}(\eps)] & = \Pr[\forall k\in\slow(\mathsf{R}): Z_k(N_c) = 0\mid \mathcal E(\eps)] + o(1),
\end{split}
\]
where the error term does not depend on $\eps$. Since $\Pr[\mathcal{E}(\eps)] = 1-o(1)$, for $n\to\infty$ the left hand side converges to $\Pr[\Tcon\leq N_c]$, while the right hand side converges to $\prod_{k\in \slow(\mathsf{R})} \exp\{-d_ke^{-c}\}$ by the method of moments, thus proving the claim.

It remains to prove \eqref{eq:fallingmoments}. For this, let $H$ be the set of all ordered tuples $((C_{k,i})_{i=1}^{e_k})_{k\in \slow(\mathsf{R})}$ of pairwise distinct components $C_{k,i}\in \comp_k\big(G_{N_\delta}^{\mathsf{R}}\big)$. For $h\in H$ we write $Z_h(N) = 1$ if $\prod_{C\in H}Z(C)=1$ or, in other words, if none of the components $C_{k,i}$ of $h$ occur in a $k$-good position between rounds $N_\delta$ and $N$. Note that an  elementary counting argument implies
\[
	\prod_{k\in\slow(\mathsf{R})}Z_k^{\underline{e_k}}(N) = \sum_{h\in H} Z_h(N)
\] 
Consider any tuple $h= ( (C_{k,i})_{i=1}^{e_k} )_{k\in \slow(\mathsf{R})}\in H$. The probability that for a fixed $k\in \slow(\mathsf{R})$ none of the components $(C_{k,i})_{i=1}^{e_k}$ appears in a $k$-good position in a given round $N$ is $(1-ke_k/n)^{\ext(\mathsf R)/k}$ (even if we condition on $\mathcal E(\eps)$). Similarly, using the fact that $\Pr[\forall i: \mathcal{A}_i] = \prod_i \Pr[\mathcal{A}_i \mid \forall j < i: \mathcal{A}_j]$, we deduce that the probability that none of the components $\{C_{k,i}\}_{1 \le i \le e_k}$ appears in a $k$-good position in a given round $N$ is
\[
	F:= \prod_{k \in \slow(\mathsf{R})}\left(1-\frac{ke_k}{n - O(1)}\right)^{\ext(\mathsf R)/k},
\]
where the $O(1)$ term depends only on $e_1, \dots, e_k$ and $\mathsf{R}$.  Thus, for every $N\geq N_\delta$,
\[
	\EE\left[\prod_{k\in\slow(\mathsf R)}Z_k^{\underline{e_k}}(N+1)\mid \mathcal E(\eps)\right]
	= F \cdot 
	\EE\left[\prod_{k\in\slow(\mathsf R)}Z_k^{\underline{e_k}}(N)\mid \mathcal E(\eps)\right].
\]
Moreover, by definition of $\mathcal{E}(\eps)$ we have 
\[\EE\left[\prod_{k\in\slow(\mathsf R)}Z_k^{\underline{e_k}}(N_\delta)\mid \mathcal E(\eps)\right] =
\prod_{k\in\slow(\mathsf R)}\left((1\pm 2\eps)d_kn^{\delta}\right)^{e_k}, \]
for all large enough $n$, and so by induction we get for any $c \in \mathbb{R}$
\[
\begin{split}
	\EE\left[\prod_{k\in\slow(\mathsf R)}Z_k^{\underline{e_k}}(N_c)\mid \mathcal E(\eps)\right]
  &= 
  F^{N_c - N_\delta} \cdot \left((1\pm 2\eps)d_kn^{\delta}\right)^{e_k}.
\end{split}
\]
Note that $1-\frac{ke_k}{n - O(1)} = ( e^{-k/n + O(n^{-2})})^{e_k}$. Thus,
\[
\begin{split}
	\EE\left[\prod_{k\in\slow(\mathsf R)} \hspace{-3mm} Z_k^{\underline{e_k}}(N_c)\mid \mathcal E(\eps)\right]
  &= \prod_{k\in \slow(\mathsf R)} \left((1\pm 2\eps)\left(e^{-k/n+O(n^{-2})}\right)^{\ext(\mathsf R)(N_c-N_\delta)/k}d_kn^{\delta}\right)^{e_k}.
\end{split}
\]
Since $\ext(\mathsf R)(N_c-N_\delta) = cn+\delta n\log n$ the claim in~\eqref{eq:fallingmoments} follows immediately.

For later reference (and omitting the details), we note that a slight variation on this argument shows the following: for every $\eps>0$ and
$c\in\mathbb R$, and for all $k\in\slow(\mathsf R)$, we have, as $n\to\infty$
\begin{equation}\label{eq:distr1} \Pr[Y_k(N_{c+\eps}) = 0 \land \forall i\neq k: Y_i(N_c) = 0]
\to e^{-d_ke^{-(c+\eps)}}\prod_{i\in \slow(\mathsf R)\setminus\{k\}} \hspace{-5mm} e^{-d_ie^{-c}}
\end{equation}
and
\begin{equation}\label{eq:distr2} \Pr[Y_k(N_{c}) = 0 \land \forall i\neq k: Y_i(N_{c+\eps}) = 0]
\to e^{-d_ke^{-c}}\prod_{i\in \slow(\mathsf R)\setminus\{k\}} \hspace{-5mm} e^{-d_ie^{-(c+\eps)}}.
\end{equation}

~\\
\emph{Proof of (c).} We consider $\Tcon'(n) = \ext({\mathsf R})\cdot \Tcon^\mathsf{R}(n)/{n}- \log n$. Let $D$ be a random variable with distribution $\Pr[D \le c] = F(c) := \prod_{k\in \slow(\mathsf{R})}\exp\{-d_ke^{-c}\}$, where $c \in \mathbb{R}$. Then, by (b), $\Tcon'(n)$ converges in distribution to $D$. In the following we will prove that  the sequence $\Tcon'(n)$ is uniformly integrable, i.e.,
\begin{equation}
\label{eq:uniformintegrable}
\limsup_{n\in\NN} \left(\EE\left[|\Tcon'(n)|_{\geq \alpha}\right]\right) \to 0 \qquad \text{as $\alpha \to \infty$,}
\end{equation}
where $X_{\geq \alpha} = X$ if $X \geq \alpha$, and $X_{\geq \alpha} =0$ otherwise.

First we show how \eqref{eq:uniformintegrable} implies the statement of (c). Convergence in distribution together with uniform integrability implies convergence of the means, i.e., $\EE[\Tcon'(n)] \to \EE[D]$ (see, e.g., \cite{BillingsleyBook}). By elementary calculus and the change of variables $u = \sum_{k\in\slow(\mathsf{R})} d_ke^{-c} = e^{c_0-c}$ we get
\begin{align*}
\EE[D] &  = \int_{-\infty}^{\infty} c \left(\prod_{k\in \slow(\mathsf{R})}e^{-d_ke^{-c}}\right)\left(\sum_{k\in\slow(\mathsf{R})} d_k e^{-c}\right) dc \\
& = \int_{\infty}^0 (-c_0+\log u)e^{-u}du
  = c_0 + \gamma,
\end{align*}
where we used the well-known identity $\gamma = -\int_{0}^{\infty}  (\log
u)e^{-u} du$ for the Euler-Mascheroni constant. Thus, $\EE[\Tcon'(n)] =
\gamma+c_0 + o(1)$, and
\[
\EE[\Tcon(n)] = \frac{n}{\ext(\mathsf{R})}\left(\EE[\Tcon'(n)]+\log n\right) = \frac{n\log n+\gamma n +c_0n}{\ext(\mathsf{R})} + o(n).
\]
Thus it suffices to prove~\eqref{eq:uniformintegrable}. We define the following events, where $\delta = 1/3$ (and, as usual, $Y_k(N)$ is the number of vertices in $k$-components of $G_N^\mathsf{R}$):
\begin{enumerate}
\item[(i)] for $N_0 = n\log n/\ext(\mathsf R)$ we have $Y_k(N_0) = 0$ for all $k\in\fast(\mathsf{R})$, and there is only one $\omega$-component in $G_{N_0}^{\mathsf{R}}$,
\item[(ii)]  $e^{c_0}n^{\delta}/2 \leq \sum_{k \in \slow(\mathsf{R})} Y_k(N_\delta)$ and $\sum_{1\leq k\leq K} Y_k(N_\delta) \leq 2e^{c_0}Kn^{\delta}$, and
\item[(iii)] all rounds between $N_{\delta}$ and $N_{\infty}$ are regular.
\end{enumerate}
Then by part (a) of Theorem~\ref{thm:main}, by Lemma~\ref{lemma:phase2}, and by
the properties of $N_\infty$, respectively, the events {\em(i)}, {\em(ii)} and
{\em(iii)} each have probability $1-o(1/\log n)$, where for \emph{(ii)} we
also use $\sum_{k\in\slow(\mathsf{R})} e^{c_k} \leq Ke^{c_0}$. For the proof we also need the following claim, whose justification we postpone to a later point: there exists a constant $\eta>0$ such that
\begin{equation}
  \label{eq:conditionalbound}
  \EE\left[|\Tcon'(n)|_{\geq \alpha}\mid \Tcon'(n)> 2 \lfloor\log\log n \rfloor\right] \leq \eta\log n.
\end{equation}
Our next goal is to give bounds for $\Pr[\Tcon'(n)>c]$ that are uniform in $c$ (in (b) we calculated the limit of this probability only for constant $c$). For every $1\leq k \leq K$ and $N\geq N_\delta$, write $X_k(N)$ for the number of components in $\comp_k(N_\delta)$ that never appear at a $k$-good position between rounds $N_\delta$ and $N$. Furthermore, let $X_{\slow}(N) := \sum_{k\in \slow(\mathsf R)} X_k(N)$. Note that for every $c$ such that $N_0\leq N_c\leq N_\infty$ we have 
\begin{equation}\label{eq:tconbounds}
\Pr[\Tcon'(n)> c] \leq \Pr[X_\slow(N_c) >0] + o(1/\log n),
\end{equation}
since $\Tcon'(n)> c$ implies that at least one of the events $\neg (i)$, $\neg (iii)$ or $X_\slow(N_c)>0$ occurs. Conversely, for every $c$ such that $N_\delta\leq N_c\leq N_\infty$ (note that the range for $c$ in~\eqref{eq:tconbounds} is different) we have
\begin{equation}\label{eq:tconbounds2}
	\Pr[\Tcon'(n)> c] \geq \Pr[X_\slow(N_c) >0] - o(1/\log n),
\end{equation}
since $X_\slow(N_c)>0$ implies that at least one of the events $\neg (iii)$ or $\Tcon'(n)> c$ occurs.

Observe that in every round a $k$-component $C$ fails to appear
at a $k$-good position with probability $(1-k/n)^{\ext_k(\mathsf R)/k}\leq (1-k/n)^{\ext(\mathsf R)/k}$. Therefore, for every $M\in \mathbb N$ and $N\geq N_\delta$,
\[
	\EE\left[X_k(N) \mid |\comp_k(G_{N_\delta}^{\mathsf R})|\le M\right]
	\leq M\cdot\left(1-\frac{k}{n}\right)^{(N-N_\delta)\ext(R)/k}.
\]
Thus, by Markov's inequality and the fact that $(ii)$ occurs with probability $1-o(1/\log n)$, we get from~\eqref{eq:tconbounds} uniformly for $c$ such that $N_0 \le N_c \le N_\infty$
\begin{equation}\label{eq:upperb}
\begin{split}
	\Pr[\Tcon'(n)> c]
	& \leq \Pr[X_\slow(N_c)> 0\mid (ii)]+o\left({1}/{\log n}\right)\\
&\leq 2K^2e^{c_0}n^{\delta}\cdot\max_{k \in \slow(\mathsf{R})}\left(1-\frac{k}{n}\right)^{(N_c-N_\delta)\ext(\mathsf{R})/k} + o\left({1}/{\log n}\right) \\
 & \leq 2K^2e^{c_0-c}+o\left({1}/{\log n}\right).
\end{split}
\end{equation}
On the other hand, a $k$-component, where $k\in\slow(\mathsf{R})$, fails to appear at a $k$-good position with probability $(1-k/n)^{\ext(\mathsf R)/k} \geq 1-\ext(\mathsf R)/n$. Therefore, in every given round $N\geq N_\delta$, the probability that $X_\slow(N)$ decreases is at most ${X_\slow (N)\ext(\mathsf R)}/{n}$, by the union bound. This allows us to couple the number of rounds until $X_\slow(N)$ decreases with geometrically distributed random variables. Indeed, for every $i\leq n/\ext(\mathsf R)$, let $T_i$ be geometrically distributed with mean ${n}/(i\ext(\mathsf R))$, and let $T = \sum_{i=1}^{e^{c_0}n^\delta/(2K)}T_i$. Then, for every $N_\delta \leq N_c \leq N_\infty$ and $x\in\mathbb N$
\[
	\Pr[X_{\slow}(N_c) = 0\mid (ii)]\leq \Pr[T \leq N_c-N_\delta].
\]
It is not difficult to bound $\Pr[T \leq N_c-N_\delta]$. Straightforward calculation show that for a suitable constant $\zeta>0$ we have that $\EE[T] \geq {(n\log n-\zeta n)}/{3\ext(\mathsf R)}$ and $\text{Var}[T] \leq \zeta {n^2}/{\ext(\mathsf R)^2}$. Hence, by \eqref{eq:tconbounds2} and Chebyshev's inequality, and since $(ii)$ occurs with probability $1-o(1/\log n)$, for all $-\delta \log n< c< -\zeta/3$
\begin{equation}
\label{eq:lowerb}
\begin{split}
\Pr[\Tcon'(n)\leq c] & \leq \Pr[X_{\slow}(N_c) = 0] + o(1/\log n)\\
&\leq \Pr\left[T \leq N_c-N_\delta\right] +o(1/\log n)\\
&=  \Pr\left[T \leq \frac{n\log n + 3c n}{3\ext(\mathsf R)}\right]+o(1/\log n) \quad [\text{as } \delta = 1/3]\\
& \leq {9\zeta}{(\zeta+3c)^{-2}} + o(1/\log n).
\end{split}
\end{equation}
We are now ready to complete the proof of \eqref{eq:uniformintegrable}. For $\alpha> 0$, we write
\begin{equation}\label{eq:split} \EE\left[|\Tcon'(n)|_{\geq \alpha}\right] = \EE\left[\Tcon'(n)_{\geq \alpha}\right]
-\EE\left[\Tcon'(n)_{\leq -\alpha}\right].
\end{equation}
We will consider each term separately. For the second term, observe that by the
definition of $\Tcon'(n)$ the inequality $\Tcon'(n)\geq -\log n$ holds. Thus, for $X = (-\Tcon')_{\geq \alpha}$ the general formula $\EE[X] = \int_0^{\infty}\Pr[X \geq c]dc$ simplifies to
\begin{align*}
  -\EE[\Tcon'(n)_{\leq -\alpha}]  
	= \alpha\Pr[\Tcon'(n) \le -\alpha] + \int_{\alpha}^{\log n}\Pr[\Tcon'(n) \le -c]dc.
\end{align*}
Since the integrand is non-increasing (as a function of $c$) we get that
\begin{align*}
  -\EE[\Tcon'(n)_{\leq -\alpha}]  
	\le \alpha\Pr[\Tcon'(n) \le -\alpha] + \sum_{c = \lfloor \alpha \rfloor}^{\lceil \log n \rceil} \Pr[\Tcon'(n) \le -c].
\end{align*}
Then, for all sufficiently large $\alpha > \zeta/3$, \eqref{eq:lowerb} gives
\begin{align*}
  -\EE[\Tcon'(n)_{\leq -\alpha}]  
	\le \frac{10\zeta}\alpha + \sum_{c = \lfloor \alpha \rfloor}^{\lceil \log n \rceil} \frac{9\zeta}{(\zeta-3c)^2} + o(1)
	\le \frac{20\zeta}\alpha + o(1).
\end{align*}
This establishes that
$\limsup_{n\to\infty}-\EE\left[\Tcon'(n)_{\leq -\alpha}\right] \to 0$ as $\alpha \to \infty$. For the first term in~\eqref{eq:split}, by \eqref{eq:conditionalbound} and \eqref{eq:upperb},
\begin{align*}
\EE\left[\Tcon'(n)_{\geq \alpha}\right]
& \leq \sum_{c = \lfloor\alpha\rfloor}^{2\lfloor \log\log n\rfloor } (c + 1)\big(2K^2e^{c_0-c}+o(1/\log n)\big) \\
& \quad + \big(2K^2e^{c_0-2\lfloor \log\log n\rfloor}+o(1/\log n)\big)(\eta+3)\log n.
\end{align*}
The second summand is in $O(e^{-2\log\log n}\log n)+o(1) = o(1)$. Moreover, the first term is a partial sum of a converging series, and becomes arbitrarily small as $\alpha\to\infty$. We obtain
\[ \limsup_{n\to\infty}\EE\left[\Tcon'(n)_{\geq \alpha}\right] \to 0 \qquad \text{ as }\alpha \to \infty. \]
This completes the proof of (c), assuming \eqref{eq:conditionalbound},
and it only remains to prove this auxiliary claim. So let $N\in \mathbb N$, and let $G_0$ be a non-empty graph on $n$ vertices. We will bound the conditional expectation $\EE[|\Tcon'(n)|  \mid G_{N}^{\mathsf{R}}=G_0]$. If $G_0$ is
connected, then $\Tcon(n) \leq N$, so assume otherwise. Fix some $\eps > 0$ with the property $(1-\eps)^{\ell-1} > 1/2$. Let $\mathcal{E}_{\text{giant}}(N')$ be the event that in round $N'$ there is a giant component with $(1-\eps)n$ vertices. Then by Lemma~\ref{lemma:universalfastgiant} there is a constant $\rho>0$ such that uniformly over all $G_0$ we have $\Pr[\mathcal{E}_{\text{giant}}(N_\rho) \mid G_{N}^{\mathsf{R}}=G_0] \geq 1-o(1/n)$, where $N_\rho := N+\rho n$.  

Assuming that $\mathcal{E}_{\text{giant}}(N_\rho)$ occured, let $S$ be the vertex set of the giant. Let $Z(N)$, $N\geq N_\delta$, be the set of vertices in $V\setminus S$ that have no neighbor in $S$ in round $N$. Fix some $v \in Z(N)$. Then the probability that in round $N+1$ the vertex $v$ appears at position $i$ (among the $\ell$ randomly selected vertices), while at all other positions there are vertices of $S$ is at least $1/n \cdot (1-\eps)^{\ell-1} > 1/(2n)$. If $v$ is in a component of size $k\leq K$, then there is a $k$-good position, so with probability at least $1/(2n)$ the vertex $v$ is joined to $S$ by an edge. Similarly, if $v$ is in a component of size larger than $K$, then it is also joined to $S$ with probability at least $1/(2n)$. Since $Z(N_\delta) \leq n/2$, 
\[
\EE[|Z(N_\rho +\Delta)| \mid \mathcal{E}_{\text{giant}}(N_\rho)\text{ and } G_{N}^{\mathsf{R}}=G_0] \leq \frac{n}{2}\left(1-\frac{1}{2n}\right)^{\Delta}
\]
for every $\Delta \in \mathbb N$. In particular, for $\Delta = 4n\log n$ the
right hand side is at most $1/(2n)$. Note that $|Z(N)| = 0$ implies $\Tcon(n)\leq N$. Thus, by Markov's inequality,
\[
\Pr[\Tcon(n) > N_\rho + 4n\log n \mid \mathcal{E}_{\text{giant}}(N_\rho)\text{ and } G_{N}^{\mathsf{R}}=G_0] \leq {1}/{2n}.
\]
Note that $N_\rho + 4n\log n < N+5n\log n$ for sufficiently large $n$.
Therefore, we get for sufficiently large $n$ (but uniformly for all $N$ and
$G_0$):
\begin{align*}
& \Pr\left[\Tcon(n) > N + 5n \log n \mid G_{N}^{\mathsf{R}}=G_0 \right] \\
& \qquad \leq \Pr\left[\neg\mathcal{E}_{\text{giant}}(N_\rho)\mid G_{N}^{\mathsf{R}}=G_0\right] \\
& \qquad \qquad + \Pr\left[\Tcon(n) > N_\rho + 4n\log n \mid \mathcal{E}_{\text{giant}}(N_\rho)\text{ and }G_{N}^{\mathsf{R}}=G_0\right] \\
&\qquad \leq {1}/{n}.
\end{align*}
Applying this bound iteratively, we find that for sufficiently large $n$ we
have for all $N >0$, all graphs $G_0$, and all $i\in\mathbb{N}$,
\begin{equation}
\label{eq:largetimes}
\Pr\left[\Tcon(n) > N + 5i n \log n \mid G_{N}^{\mathsf{R}}=G_0\right] \leq n^{-i}.
\end{equation}
In particular, $\EE\left[\Tcon(n) \mid G_{N}^{\mathsf{R}}=G_0\right] \leq
N+n\log n \cdot O(\sum_{i=0}^{\infty}i\cdot n^{-i}) = N+O(n\log n)$. Recalling
the definition $\Tcon'(n) = {\ext({\mathsf R})\Tcon(n)}/{n} - \log n$,
we get $\EE\left[\Tcon'(n) \mid G_{N}^{\mathsf{R}}=G_0\right] \leq
{\ext({\mathsf R})\cdot N}/{n}+O(\log n)$. On the other hand, $\Tcon'(n)\geq - \log n$ always, since $\Tcon(n)\geq 0$. Summarizing, there exists $\eta' > 0$ such that for all $n\geq 2$, all non-empty graphs $G_0$, and all $N >0$,
\begin{equation}
\label{eq:largetimes2}
\EE\left[|\Tcon'(n)| \mid G_{N}^{\mathsf{R}}=G_0\right] \leq \frac{\ext({\mathsf R})\cdot N}{n}+ \eta' \log n.
\end{equation}
Let us denote by $\cal L$ the event that $\Tcon'(n) > 2\log\log n$, and let $N' = \lceil(n\log n +2n\lfloor\log\log n\rfloor)/\ext(\mathsf{R})\rceil$. Then,
\begin{align*}
	\EE\left[|\Tcon'(n)|_{\geq \alpha}\mid {\cal L}\right]
	& = \sum_{G_0}
		\Pr\left[G_{N'}^{\mathsf{R}} = G_0\mid {\cal L}\right] \cdot \EE\left[|\Tcon'(n)|_{\geq \alpha}\mid G_{N'}^{\mathsf{R}}=G_0 \right] \\
	& \leq \left(\frac{\ext({\mathsf R})\cdot N'}{n}+ \eta' \log n \right) \sum_{G_0}\Pr\left[G_{N'}^{\mathsf{R}} = G_0\mid {\cal L}\right]\\
& = {\ext({\mathsf R})\cdot N'}/{n}+ \eta' \log n,
\end{align*}
and~\eqref{eq:conditionalbound} follows.

\vspace{10pt}
\noindent\emph{Proof of (d).}
For a $C>0$ let $\mathcal{A}=\mathcal{A}(C)$ be the event that 
\begin{itemize}
	\item all rounds from $N_{-C}$ to $N_C$ are regular (where $N_{-C}$ and $N_C$ are defined in the beginning of the proof),
	\item there is only one $\omega$-component in $G_{N_{-C}}^\mathsf{R}$ and
	\item for all $k\in\fast(\mathsf{R})$ we have $Y_k(N_{-C}) = 0$.
\end{itemize} 
Fix $k\in
\slow(\mathsf{R})$ and  $\varepsilon>0$  and define for all $-C\le c\le C$  three events
\[
  \mathcal{E}_1(c) := \mathcal{A}   ~\wedge~ Y_k(N_{c+\varepsilon})=0  ~\wedge ~ Y_k(N_c) > 0 ~\wedge ~  \forall i\in\slow(\mathsf{R})\setminus\{k\}: Y_i(N_c) = 0 ,
\]
and
\[
\begin{aligned}
\mathcal{E}_2(c) &:= &\mathcal{A}&~\wedge~ T_k^\mathsf{R} = \Tcon^\mathsf{R} && \wedge~\Tcon^\mathsf{R}\in(N_c,N_{c+\varepsilon}],\\
\mathcal{E}_3(c) & :=&\mathcal{A}&~\wedge~ Y_k(N_c) > 0 && \wedge~\forall i\in\slow(\mathsf{R}): \ Y_i(N_{c+\varepsilon}) = 0.
\end{aligned}
\]
Since $\cal A$ guarantees that all rounds from $N_{-C}$ to $N_C$ are regular, we have that $\mathcal{E}_1(c)$ implies $\mathcal{E}_2(c)$ which in turn implies $\mathcal{E}_3(c)$. That is, we have
\begin{equation}\label{eq:eventbounds}
  \Pr[\mathcal{E}_1(c)] \leq \Pr[\mathcal{E}_2(c)] \leq \Pr[\mathcal{E}_3(c)]\qquad\text{for all } -C \leq c\leq C-\varepsilon.
\end{equation}
From \eqref{eq:distr1} and \eqref{eq:distr2}, and the fact that $\Pr[\mathcal A] =1-o(1)$, we infer that for all $-C \leq c\leq C-\varepsilon$ we have
\begin{align}
  \lim_{n\to\infty}\Pr[\mathcal{E}_1(c)] &= \left(e^{-d_ke^{-(c+\varepsilon)}}-e^{-d_ke^{-c}}\right)\prod_{i\in\slow(\mathsf{R})\setminus\{k\}}e^{-d_ie^{-c}},\nonumber\\
  \lim_{n\to\infty}\Pr[\mathcal{E}_3(c)] &= \left(e^{-d_ke^{-(c+\varepsilon)}}-e^{-d_ke^{-c}}\right)\prod_{i\in\slow(\mathsf{R})\setminus\{k\}}e^{-d_ie^{-(c+\varepsilon)}}\label{eq:bounds}.
\end{align}
Let $f_d(c) = e^{-de^{-c}}$. Then we infer, with $S_k :=
\sum_{i\in\slow(\mathsf{R})\setminus\{k\}}d_i$, by applying Taylor's theorem
\begin{align*}
  \lim_{n\to\infty}\Pr[\mathcal{E}_1(c)] &= (\varepsilon f'_{d_k}(c) + O(\varepsilon^2))\cdot f_{S_k}(c),\text{~and}\\
  \lim_{n\to\infty}\Pr[\mathcal{E}_3(c)] &= (\varepsilon f'_{d_k}(c) + O(\varepsilon^2))\cdot f_{S_k}(c+\varepsilon).
\end{align*}
Note that since $f_d$ is smooth, and since we will be considering
only values of $f_d$ and its derivatives in a compact interval $[-C,C]$ for a $C>0$ independent of $\varepsilon$, there exists is  universal constant
$C'$ (depending on $C$ only) such that all error terms are in absolute value at most $(C'-1)\varepsilon^2$. Moreover, let $S_{C,\eps}:= \{ j \cdot \varepsilon \mid j\in \N,  -C \leq  j \cdot \varepsilon \leq C-\varepsilon\}$. Since $|S_{C,\eps}|$ is a constant, for sufficiently large $n$ the probabilities $\Pr[\mathcal{E}_1(c)]$ and $\Pr[\mathcal{E}_3(c)]$ are within distance at most $\eps^2$ from their respective limits for all $c\in S_{C,\eps}$. Therefore, together with~\eqref{eq:eventbounds} we obtain that there is $C'>0$ and $n_0\in\mathbb{N}$ such that for all $n \geq n_0$ and all $c\in S_{C,\eps}$ 
\begin{equation}\label{eq:limitineq}
  \varepsilon f'_{d_k}(c) f_{S_k}(c) + C'\varepsilon^2 \leq \Pr[\mathcal{E}_1(c)] \leq  \Pr[\mathcal{E}_3(c)]\leq\varepsilon f'_{d_k}(c) f_{S_k}(c+\varepsilon) + C'\varepsilon^2.
\end{equation}
With those preparations at hand,  let $\mathcal{E}^*$ be the event that $T_k^\mathsf{R} = \Tcon^\mathsf{R}$ and $\Tcon^\mathsf{R} \in (N_{-C},N_{C}]$ and that $\mathcal{A}$ holds.  We may assume that $C$ is a multiple of $\varepsilon$. Then 
\[
  \Pr[\mathcal{E}^*] = \sum_{j = -C/\varepsilon}^{(C-\varepsilon)/\varepsilon}\Pr[\mathcal{E}_2(j\varepsilon)].
\]
For any $\varepsilon'>0$ we have, by choosing $C = C(\varepsilon')$ large enough,
that
\begin{align*}
  & |\Pr[T_k^\mathsf{R} = \Tcon^\mathsf{R}] - \Pr[\mathcal{E}^*]|\\
  &= |\Pr[T_k^\mathsf{R}=\Tcon^\mathsf{R}\land T_k^\mathsf{R}\not\in (N_{-C},N_C]] + \Pr[T_k^\mathsf{R} = \Tcon^\mathsf{R} \land T_k^\mathsf{R}\in(N_{-C},N_C]\land\lnot\mathcal{A}]| \\
  &\leq \Pr[\Tcon^\mathsf{R}\not\in (N_{-C},N_C]] + \Pr[\lnot\mathcal{A}].
\end{align*}
However, the last expression is at most $\varepsilon'$, due to part (b) for $C(\varepsilon')$ and $n = n(\varepsilon')$ large enough. In particular, this derivation, combined with~\eqref{eq:eventbounds}, implies that
\begin{equation*}\label{eq:sandwichpartc}
  \sum_{j = -C/\varepsilon}^{(C-\varepsilon)/\varepsilon}\Pr[\mathcal{E}_1(j\varepsilon)]
  \leq \Pr[T_k^\mathsf{R} = \Tcon^\mathsf{R}] \leq \varepsilon' + \sum_{j =
  -C/\varepsilon}^{(C-\varepsilon)/\varepsilon}
  \Pr[\mathcal{E}_3(j\varepsilon)].
\end{equation*}
Thus~\eqref{eq:limitineq} guarantees that for sufficiently large $n$,
\[
  \left(\sum_{j =  -C/\varepsilon}^{(C-\varepsilon)/\varepsilon}\varepsilon f'_{d_k}(j\varepsilon) f_{S_k}(j\varepsilon)\right) + C'C\varepsilon
  \leq \Pr[T_k^\mathsf{R} = \Tcon^\mathsf{R}]
\]
and
\[ 
  \Pr[T_k^\mathsf{R} = \Tcon^\mathsf{R}]
  \leq \varepsilon' + \left(\sum_{j =  -C/\varepsilon}^{(C-\varepsilon)/\varepsilon}
  \varepsilon f'_{d_k}(j\varepsilon) f_{S_k}((j+1)\varepsilon)\right) + C'C\varepsilon.
\]
Since the statement holds for any choice of $\varepsilon>0$ we have
\[
\int_{-C}^C f_{S_k}(x)\cdot f_{d_k}'(x)dx \leq \lim_{n\to\infty}\Pr[T_k^\mathsf{R} = \Tcon^\mathsf{R}]\leq \varepsilon' + \int_{-C}^C f_{S_k}(x)\cdot f_{d_k}'(x)dx.
\]
Again this statement holds for any choice of $\varepsilon'>0$ if $C = C(\varepsilon')$ is large enough. Hence,
\[
  \lim_{n\to\infty}\Pr[T_k^\mathsf{R} = \Tcon^\mathsf{R}] = \lim_{C\to\infty}\int_{-C}^C f_{S_k}(x)\cdot f_{d_k}'(x)dx.
\]
Setting $S = S_k +d_k=\sum_{i\in\slow(\mathsf R)} d_i$, the integral can be computed as follows:
\begin{align*}
  \lim_{C\to\infty}\int_{-C}^C f_{S_k}(x) f_{d_k}'(x)dx &= \lim_{C\to\infty}\int_{-C}^C e^{-S_ke^{-x}}\cdot e^{-d_ke^{-x}}\cdot d_ke^{-x} dx \\
  &= \lim_{C\to\infty}d_k\int_{-C}^C e^{-S e^{-x} - x}dx \\
  & = \lim_{C\to\infty}\frac{d_k}{S}\left( e^{-S e^{-C}} - e^{-S e^{C}} \right)
  = \frac{d_k}{S},
\end{align*}
as claimed.
\end{proof}


\section{Degenerate rules}\label{sec:degenerate}
In this section we will discuss lower bounds for degenerate rules and we will prove Theorem~\ref{thm:degenerate}.  As an auxiliary result we need the following lemma implying that the typical behavior of the process is to end up with fast components gone and only a constant number of slow components left. For brevity we will denote $\fast(\mathsf{R})$ and $\slow(\mathsf{R})$ by $\fast$ and $\slow$, respectively. Moreover, we will denote $\sum_{k\in\fast}Y_k(N)$ by $Y_\fast(N)$, $\sum_{k\in\fast}z_k(t)$ by $z_\fast(t)$, $\sum_{k\in\slow}Y_k(N)$ by $Y_\slow(N)$ and $\sum_{k\in\slow}z_k(t)$ by $z_\slow(t)$. Recall that a component is \emph{small} if it has size $\leq K$.
\begin{lemma}
\label{lem:degenerate1}
Let $K, \ell\in \mathbb{N}$ and let $\mathsf{R}$ be a degenerate $(K,\ell)$-rule. Then for every $\eps >0$ there is a $C> 0$ and $t_0>0$ such that whp $Y_\fast(tn) < \eps Y_\slow(tn) +C\log n$ for all $t_0 < t < n$.
\end{lemma}

\begin{proof}
Let $\eps>0$. Since the statement becomes stronger for smaller $\eps > 0$ we may assume $\eps < (32\ell^2K)^{-1}$ and $(1-\eps^2)^{\ell-1} \geq 1/2$. By Lemma~\ref{lem:limits} and Lemma~\ref{lemma:universalfastgiant} there is $t_0>0$ such that $z_\fast(t) < \eps/2 \cdot z_\slow(t)$ and $z_\fast(t) + z_\slow(t) < \eps^2/2$ for all $t\in [t_0,t_0+1]$, and by
Corollary~\ref{cor:linearregime}, whp
\begin{equation}
\label{eqdeg:induction1}
Y_\fast(tn) < \eps\cdot Y_\slow(tn)\qquad \text{and} \qquad Y_\fast(tn) +Y_\slow(tn) < \eps^2 n
\end{equation}
for all $t\in[t_0,t_0+1]$.

We show by induction on $t$ that \eqref{eqdeg:induction1} holds whp 
for all $t'\in [t_0,t]$ as long as $t<n$ and $Y_\slow(tn) \geq 64K/\eps \cdot
\log n$. More precisely, we show
that if \eqref{eqdeg:induction1} holds for some $t\geq t_0$, and if
$Y_\slow(tn) \geq 64K/\eps\cdot\log n$, then with probability $1-o(1/n)$ it
also holds for $t+1$. The idea of the inductive step is similar as in the proof for Lemma~\ref{lem:limits} (a), but here we need to work with $Y_k$ instead of $z_k$, which forces us to split the proof into small steps.

Since $Y_\fast(N)+Y_\slow(N)$ is non-increasing we only need to show the first inequality of~\eqref{eqdeg:induction1}. Note that if $Y_\slow(tn) +Y_\fast(tn) < C \log n$ for some $C>0$ then the statement is trivial. So assume that \eqref{eqdeg:induction1} holds for some $t\geq t_0$, and that $Y_\slow(tn) \geq 64K/\eps\cdot\log n$. We first give a lower bound for $Y_\slow((t+1)n)$. As $\mathsf{R}$ is degenerate small components can only be removed in rounds for which at least two of the $\ell$ vertices belong to small components. Let $I$ denote the number of such rounds in the interval $(tn,(t+1)n]$.  The probability that a single round contains at least two vertices in small components is at most
 \begin{equation}
\label{eq:definitionpsmall}
 \binom{\ell}{2}\left(\frac{Y_\slow(tn) + Y_\fast(tn)}{n}\right)^2 \leq \ell^2 \frac{Y_\slow(tn)^2}{n^2}.
\end{equation}
Since the number of small components is non-increasing, using the induction assumption \eqref{eqdeg:induction1} for $tn$ guarantees that 
\begin{equation}
\label{eqdeg:numsmall}
\EE[I] \leq n\cdot \ell^2 Y_\slow(tn)^2/n^2  \leq
\ell^2 \eps^2 Y_{\slow}(tn) < \eps Y_{\slow}(tn)/(32K).
\end{equation}
Note that the right hand side
is $> 2\log n$ by our assumption on $Y_\slow(tn)$. By the Chernoff
bounds, with probability $1-o(1/n)$ the actual number of non-regular rounds is at most $I \leq \eps Y_{\slow}(tn)/(16K)$. Since in each
round at most two slow components can be merged, each having at most $K$ vertices, with probability $1-o(1/n)$
we have with room to spare
\begin{equation}
\label{eqdeg:lowerbound}
Y_\slow((t+1)n) \geq Y_\slow(tn) - 2K\cdot I
> \frac{15}{16}\cdot Y_\slow(tn).
\end{equation}
Next we derive an upper bound for the fast components. A new fast component can only be created  by merging two small components. Hence, the number of vertices in fast components that are created between time $t$ and $t+1$ is at most $K\cdot I$. To use this fact, we distinguish two cases. First assume that there exists a round $N \in (tn, (t+1)n]$ such that $Y_\fast(N) \leq \eps Y_\slow(tn)/2$. In this case, we can directly bound
\begin{equation}
\label{eqdeg:upperbound1}
Y_\fast((t+1)n) \leq Y_\fast(N)+K\cdot I
< \frac{15}{16}\cdot \eps Y_\slow(tn),
\end{equation}
where we used the bound $I \leq \eps Y_{\slow}(tn)/(16K)$, which holds with probability $1-o(1/n)$.

Now let us turn to the the second case. So assume that for all $N \in (tn, (t+\delta)n]$ we have $Y_\fast(N) \geq \eps Y_\slow(tn)/2$. Recall that for each $k\in \fast$ we have $\ext_k(\mathsf{R}) > 0$.  In other words, for each fast component $C$ there exists at least one ``good'' position so that if $C$ appears in this position (and all other positions are filled with vertices from  $\omega$-components) then $C$ is merged with an $\omega$-component. Therefore, the probability that in a fixed round $N$ a fast component is merged with an $\omega$-component is at least
\[ 
 \frac{Y_{\fast}(N)}{n}\left(1-\eps^2\right)^{\ell-1} \geq \frac{\eps Y_\slow(tn)}{4n}.
\]
Again we apply the Chernoff bounds and use that the right hand side is $> 2\log n /n$. Thus, with probability $1-o(1/n)$ the number $Z$ of fast components that are merged with an $\omega$-component between time $t$ and $t+1$ is at least $Z \geq \eps Y_\slow(tn)/8$.
Therefore, with probability at least $1-o(1/n)$,
\begin{align}
\label{eqdeg:upperbound2}
 \nonumber Y_\fast((t+1)n) & \leq Y_\fast(tn) - Z + K\cdot I \\ & \leq \eps Y_\slow(tn) - \frac{\eps Y_\slow(tn)}{8} + \frac{\eps Y_\slow(tn)}{16}  \leq \frac{15}{16}\cdot \eps Y_\slow(tn),
\end{align}
so we get the same bound as in the first case, cf.\ \eqref{eqdeg:upperbound1}.

In either case, together with~\eqref{eqdeg:lowerbound} the inductive conclusion follows since with probability $1-o(1/n)$,
\[
\eps Y_\slow((t+1)n) \geq \frac{15}{16}\eps Y_\slow(tn) \geq Y_\fast((t+\delta)n).
\]
This concludes the induction step and the proof of the lemma.
\end{proof}

\begin{proof}[Proof of Theorem~{\ref{thm:degenerate}}]
Choose any $0<\eps<1$, and let $C,t_0>0$ be as in Lemma~\ref{lem:degenerate1}.
From Lemma~\ref{lem:propdiffeq} and Corollary~\ref{cor:linearregime} we know
that at time $t_0$ there is a linear number of slow components. We distinguish
two cases. First, if the number of vertices in small components is larger than $Y_0 := (C+1+\eps) \cdot \log n$ after $n^2$ rounds, then there is nothing to show. Secondly, suppose that the number drops below $Y_0$ at some round $N_0$. Then Lemma~\ref{lem:degenerate1} implies $Y_\slow(N_0) \geq \log n$, while $Y_\fast(N_0) \leq Y_0 = O(\log n)$. We wait a bit further until in some round $N_1$ the number of slow
components has dropped to $1$ or $2$. (In each round, it can decrease by at
most $2$.) Then the number of fast components is still at most $Y_0$.

Now we wait for $\Delta := n^{3/2}$ further rounds. The probability that in a fixed round at least two vertices in small components are chosen is in $O((Y_0/n)^2) = O(\log^2 n/n^2)$. Thus, the probability that there exists a round between $N_1$ and $N_1+\Delta$ in which two vertices in small components are chosen is in $O(\log^2 n/n^2 \cdot \Delta) = o(1)$. In particular, whp the set of slow components remains unchanged.

On the other hand, in each round the probability that a particular fast component is merged with an $\omega$-component is $\Omega(1/n)$. Thus, the expected number of fast components that will remain after $\Delta$ rounds is $(1-\Omega(1/n))^{\Delta} = o(1)$. By Markov's inequality, the probability that there are no fast components left is $1-o(1)$. Thus whp after $N_1+\Delta$ rounds there is no fast component left, and only one or two slow components. Then we only merge the slow components if at least two of their vertices are selected in some round. The expected time until this happens is $\Omega(n^2)$, which proves the theorem.
\end{proof}


\section{Examples and Applications}

In this section we give some examples that illustrate Theorem~\ref{thm:main}. In Section~\ref{ssec:LEX} we describe a rule that is asymptotically fastest to connect the graph among all Achlioptas processes.

\subsection{The BF-Process}
\label{sec:BohmanFrieze}

We consider the Bohman-Frieze process~\cite{bohman2001avoiding}. In each round
we are given two edges (so $\ell=4$), and choose according to the following
rule. If the first edge connects two isolated vertices, then it is added to the
graph. Otherwise, we choose the second edge. Denote the choice rule of the
Bohman-Frieze process by $\mathsf{BF}$. From the definition it follows
immediately that $\mathsf{BF}$ is a (1,4)-rule. By using $*$ as a placeholder
for either $1$ or $\omega$ we have that the $\mathsf{BF}$ rule maps component
size vectors of the form $(1,1,*,*)$ to $1$ and all other vectors to $2$.
Components can be combined in three different ways given by
\begin{align*}
  C_{1,1} &= \{(1,1,*,*),(\omega,*,1,1),(1,\omega,1,1)\},\\
  C_{1,\omega} &= \{(\omega,*,1,\omega),(\omega,*,\omega,1),(1,\omega,1,\omega),(1,\omega,\omega,1)\}\text{~and}\\
  C_{\omega,\omega} &= \{(*,\omega,\omega,\omega),(\omega,1,\omega,\omega)\}.
\end{align*}
The extinction rate for 1-components is
$$\ext(\mathsf{BF}) = \ext_1(\mathsf{BF}) =
1\cdot|\{(\omega,\omega,\omega,1),(\omega,\omega,1,\omega)\}| = 2.$$
Since $\ext_1(\mathsf{BF}) < 2K+2 = 4$, Theorem~\ref{thm:main} is applicable. As $K=1$ we have $z_1+z_\omega = 1$, so we can express the functions $f_k$ in the differential equations~\eqref{eq:diffeq3} in terms of $z_1$ only. By writing $z$ instead of $z_1$ they are given by
\[
z' = -2z- 2z^2+2z^3 \qquad \text{ and } \qquad z_{\omega}' = 2z+ 2z^2-2z^3.
\]
By integrating we get
\begin{align*}
-2T & = \int_{0}^{T}\frac{z'(t)}{z(t)+z(t)^2-z(t)^3}dt \stackrel{x = z(t)}{=}
\int_{1}^{z(T)}\frac{1}{x+x^2-x^3}dx \\
&= \left[\log x - \frac{5-\sqrt{5}}{10}\log (1+\sqrt 5 -2x)-
  \frac{5+\sqrt{5}}{10}\log (-1+\sqrt 5 +2x)\right]_{x=1}^{z(T)}.
\end{align*}
For $T \to \infty$ we know $z(T)\to 0$, so $2T+\log(z(T))$ converges to
\[
c_1
= \frac{\log(5-\sqrt{5})-\log(5+\sqrt{5})}{\sqrt{5}}
= \frac{-2 \log \varphi}{\sqrt{5}}
= -0.43040894\dots ~,
\]
where $\varphi = (1+\sqrt{5})/2$ is the golden ratio. Hence, by Theorem~\ref{thm:main}, since $c_0 = c_1$, the expected time until the graph is connected is
\[
\EE[\Tcon^\mathsf{BF}] = \frac{n\log n+\gamma n + c_1n}{\ext(\mathsf{BF})} + o(n) = \frac{n\log n+0.1468067\dots  n}{2} + o(n),
\]
and for all $c\in \RR$, since $d_1 = e^{c_1} = 0.6502431\dots$,
\[
  \lim_{n\to\infty}\Pr\left[\Tcon^\mathsf{BF} \leq \frac{n\log n +cn}{2}\right] =  e^{-d_1e^{-c}} = e^{-0.6502431\dots e^{-c}}.
\]

\subsection{The KP Process}
\label{ssec:KP}

In this section we study the $\KP$ process \cite{ar:kpXX}. The process starts with the empty graph. In each
round two edges are given (so, again, $\ell = 4$) and we choose the first one
if and only if at least one of its endpoints is an isolated vertex. Then $K=1$ and
$\ext(\KP) = 4$, as
\begin{align*}
  \ext_1(\KP) &= 1\cdot |\{(1,\omega,\omega,\omega),(\omega,1,\omega,\omega),(\omega,\omega,1,\omega),(\omega,\omega,\omega,1)\}| = 4.
\end{align*}
Moreover, $\slow(\KP) = \{1\}$. Note that Theorem~\ref{thm:main} is not directly applicable, since $\ext(\KP) \ge 2K+2$. However, we can study instead a (2,4)-rule $\KP'$ such that $G_N^{\KP} = G_N^{\KP'}$ with probability 1, as described in Remark~\ref{rem:extension}; $\KP'$ makes exactly the same choice as $\KP$ when presented the component sizes of four randomly selected vertices. Then we have $\ext_1(\KP) = \ext_1(\KP') = 4$ and
\[
  \ext_2(\KP') = 2\cdot |\{(\omega,\omega,2,\omega),(\omega,\omega,\omega,2)\}| = 4,
\]
and thus $\ext(\KP') = 4$ and $\slow(\KP') = \{1,2\}$.
Since $\ext(\KP') < 2\cdot 2 + 2$ Theorem~\ref{thm:main} is applicable to $\KP'$. As before, if we use $*$ as a placeholder for $1, 2$ and $\omega$, from the definition it follows that $\KP'$ maps component size vectors of the form
$(1,*,*,*)$, $(2,1,*,*)$ and $(\omega,1,*,*)$ to $1$ and all other component size vectors to $2$. Additionally, to describe $C_{\mu,\nu}$, let $\bullet$ be a placeholder for $2$ or $\omega$. Then, for $\KP'$
\begin{align*}
  C_{1,1} &= \{(1,1,*,*),(\bullet,\bullet,1,1)\}, \\
  C_{1,2} &= \{(1,2,*,*),(2,1,*,*),(\bullet,\bullet,1,2),(\bullet,\bullet,2,1)\}, \\
  C_{1,\omega} &= \{(1,\omega,*,*),(\omega,1,*,*),(\bullet,\bullet,1,\omega),(\bullet,\bullet,\omega,1)\}, \\
  C_{2,2} &= \{(\bullet,\bullet,2,2)\}, \\
  C_{2,\omega} &= \{(\bullet,\bullet,2,\omega),(\bullet,\bullet,\omega,2)\}\text{~and} \\
  C_{\omega,\omega} &= \{(\bullet,\bullet,\omega,\omega)\}.
\end{align*}
Recall that $z_1+z_2+z_\omega \equiv 1$. We can express $f_k$ in \eqref{eq:diffeq3} for $k\in\{1,2,\omega\}$ in terms of $z_1$ and $z_2$ only. The differential equations  are given by
\begin{align*}
z_1' =& -4z_1 + 4z_1^2-2z_1^3, \\
z_2' =& 2 z_1^4-4 z_1^3-4 z_1^2 z_2+4 z_1^2+4 z_1 z_2-4 z_2.
\end{align*}
Like in the Bohman-Frieze process we get for $z_1$
\begin{equation}
\label{eq:z1}
\begin{split}
-4T & = \int_{0}^{T}\frac{z_1'(t)}{z_1(t)-z_1(t)^2+z_1(t)^3/2}dt \stackrel{x = z_1(t)}{=} \int_{1}^{z_1(T)}\frac{dx}{x-x^2+x^3/2} \\ &= \left[-\frac{1}{2} \log \left(x^2-2 x+2\right)+\log (x)-\tan ^{-1}(1-x))\right]_{x=1}^{z_1(T)}.
\end{split}
\end{equation}
For $T \to \infty$ we know $z_1(T)\to 0$, so the expression $4T+\log(z_1(T))$ converges to 
\[
c_1 = (\pi +\log (4))/4 = 1.13197..\;.
\]
We can also compute the value of $c_2$. Note that the differential equation for $z_2$ is linear and we can rewrite it to
\[
	z_2' = f + g z_2,
	\quad \text{where} \quad
	f = 2z_1^4 - 4z_1^3 + 4z_1^2
	~~\text{and}~~
	g = -4z_1^2 + 4z_1 - 4.
\]
Thus, we can solve explicitly for $z_2$ in terms of $z_1$, and since $z_2(0)=0$
\begin{equation}
\label{eq:expz2} 
\begin{split}
	z_2(T)
	& = 
	\exp\left\{\int_0^T g(t) dt\right\}
		\cdot \int_0^T f(t)\exp\Big\{-\int_0^t g(y) dy\Big\} dt\\
	& = 	\exp\left\{-4T + 4\int_0^T (z_1(t) - z_1^2(t))  dt\right\} \\
	& \qquad\qquad\qquad	\cdot \int_0^T f(t)\exp\Big\{4t + 4\int_0^t(z_1(y)^2 - z_1(y)) dy\Big\} dt.
\end{split}
\end{equation}
In order to simplify this expression, note that
\[
	\int_0^T z_1(t) dt
	\stackrel{(x = z_1(t))}= 
	\int_1^{z_1(T)} \frac{x}{-4x + 4x^2 - 2x^3}dx
	= \frac12 \tan^{-1}(1 - z_1(T)),
\]
and the same change of variables yields
\[
\begin{split}
	\int_0^T z_1(t)^2 dt
	&= 
	\int_1^{z_1(T)} \frac{x^2}{-4x + 4x^2 - 2x^3}dx \\
	&= -\frac14 \log\left(z_1(T)^2 - 2z_1(T) + 2\right) + \frac12 \tan^{-1}(1 - z_1(T)).
\end{split}
\]
By plugging this into~\eqref{eq:expz2} and using that $z_1(T) \to 0$ we infer that as $T\to\infty$
\[
	4T + \log z_2(T) \to \log(2) + \log\int_0^\infty f(t)\exp\big\{4t-\log\big(z_1(t)^2 - 2z_1(t) + 2\big)\big\}dt .
\]
Using once more the change of variables $x = z_1(t)$ and~\eqref{eq:z1} we infer that
\[
  c_2 = \log(2) + \log\left[\int_0^1 \frac{e^{\tan^{-1}(1-x)}}{\sqrt{x^2-2x+2}}dx\right].
\]
The last integral can be approximated numerically and we get $c_2 = 1.008..$. Hence, by Theorem~\ref{thm:main}, for all $c\in \RR$, since $d_1 = e^{c_1} = 3.1017\dots$ and $d_2 = e^{c_2}/2 = 1.3700\dots$,
\[
	\lim_{n\to\infty}\Pr\left[\Tcon^\KP \leq \frac{n\log n +cn}{4}\right]
	= e^{-(d_1+d_2)e^{-c}} = e^{-4.47.. e^{-c}},
\]
and moreover,
\[
	\EE\left[\Tcon^{\KP}\right]
	= \frac{n\log n+\gamma n +\log(d_1+d_2)n}{4} + o(n) = \frac{n\log n+2.075.. n}{4} + o(n).
\]
Finally, we obtain that with some probability that is bounded away from zero and from one, the graph gets connected when the last isolated vertex/isolated edge dissapears. More precisely,
\[
  \lim_{n\to\infty}\Pr\left[T_1^{\KP} = \Tcon^{\KP}\right] = 0.693\dots \quad\text{and}\quad\lim_{n\to\infty}\Pr\left[T_2^{\KP} = \Tcon^{\KP}\right] = 0.306\dots~.
\]
Note that the same statements are also true for $\KP'$. 

\subsection{The Lexicographic Method}
\label{ssec:LEX}

Fix some even $\ell\geq 0$, and let $K\geq \ell/2$. We consider the $(K,\ell)$-rule $\Rlex$ that greedily takes smallest components first.  More precisely, we begin with mapping the component size vector to one in which the component sizes of the endpoints  are ordered, i.e.\ 
\begin{align*}
(s_1,\ldots,s_{\ell}) & \mapsto (s_1',\ldots,s_{\ell}') \\
& := \left(\min\{s_1,s_2\},\max\{s_1,s_2\}, \ldots, \min\{s_{\ell-1},s_{\ell}\},\max\{s_{\ell-1},s_{\ell}\}\right).
\end{align*}
We then choose that $i\in\{1,\ell/2\}$ for which $(s_{2i-1}',s_{2i}')$ is minimal with respect to the lexicographical ordering. In case of ties we choose the smallest eligible $i$. As an example consider $\ell=4$ and $K=2$. In this case we almost get the
$\mathsf{KP}$-rule, except that we choose the second edge (instead of the
first) for the two vectors $(1,\omega,1,1)$ and $(\omega,1,1,1)$.

Note that for all component size vectors of the form $(\omega,\ldots,\omega,k,\omega,\ldots,\omega)$, where $k\in[K]$, $\Rlex$ selects the edge with the  component of size $k$, thus we have that $\ext(\Rlex)=\ell$ and $\slow(\Rlex) = \{1\}$. Note that the condition on $K$ ensures that we may apply Theorem~\ref{thm:main}. 

We abbreviate again $z:= z_1$. Then by using $\sum_{k\in S_K} z_k=1$ we will be able to express the differential equation~\eqref{eq:diffeq3} for $z$ without reference to the other functions $z_k$, $k\in S_K$. The differential equation for $z$ is given by $\frac{dz}{dt} = -f_1^{-}(z)$. Recall that $P_{1,1}$ corresponds to the probability of adding an edge that joins two isolated vertices conditioned on the fraction of isolated vertices being $z$. Similarly, $\sum_{\mu\in S_k\setminus\{1\}}
P_{\mu,1}$ corresponds to the probability to add an edge that joins an isolated vertex to a component with at least two vertices. Thus,
\begin{align*}
  z_1' &= -f_1^{-}(z) = -2P_{1,1} - \sum_{\mu\in S_k\setminus\{1\}} P_{\mu,1}\\
  &= -2\left(1-(1-z^2)^{\ell/2}\right) - \left(1 - (1-z)^{\ell} - \left(1-(1-z^2)^{\ell/2}\right)\right)\\
  &= -2 + \left(1-z^2\right)^{\ell/2} + (1-z)^\ell.
\end{align*}
By the same means as in the Bohman-Frieze process, an explicit expression for
$c_1 = \lim_{t\to \infty}(\ell\cdot t+\log z(t))$ is given by
\[
  \lim_{z\to 0}\left(\log z + \ell\int_{1}^z \frac{1}{-2+(1-x^2)^{\ell/2}+(1-x)^\ell}dx\right).
\]
In general the integral can be expressed as a rational function in the roots of
the polynomial $f_1^{-}$ and their logarithms. Since the rational functions give little insight even for small values of $\ell$, we only give a table with the numerical values.
\[
\begin{array}{l||c|c|c|c|c|c|c|c}
\ell       & 2  & 4              & 6               & 8               & 10            & 12           & 14              & 16   \\
\hline
c_1 &  0  & 0.935.. & 1.910.. & 2.905.. & 3.912.. & 4.927.. & 5.948..  & 6.972..  \\
\hline
d_1 & 1 & 2.549.. & 6.756.. & 18.275.. & 50.03.. & 138.07.. & 383.0.. & 1.06..\cdot 10^3
\end{array}
\]
Mind that the table only gives second order terms. Since the dominating term of $\EE[\Tcon^\lex]$ is $n\log n/\ell$, the lexicographic rules become faster to connect the graph as $\ell$ increases. 

The main reason for studying this class of rules is that they provide a lower bound for $\Tcon^\mathsf{A}$ for \emph{any} $\ell$-Achlioptas process $\mathsf{A}$, not just for $(K,\ell)$-rules, in the following sense. Consider any $\ell$-Achlioptas process $\mathsf{A}$, i.e., a process in which $\ell$ vertices are drawn uniformly at random, and then any strategy may be used to choose between the $\ell/2$  edges.
We claim that for every $N\geq 0$ the number of isolated vertices after $N$ rounds of $\mathsf{A}$ stochastically dominates the number of isolated vertices after $N$ rounds of $\Rlex$. Formally, if $Y_{1}^{\text{lex}}(N)$ and $Y_{1}^{\mathsf{A}}(N)$ denote the number of isolated vertices after $N$ rounds
of $\Rlex$ and of $\mathsf{A}$, respectively, then for every $N\geq 0$ and every $\mu\in \NN_0$,
\begin{equation}
\label{eq:domination}
\Pr[Y_{1}^{\text{lex}}(N) \leq \mu] \geq \Pr[Y_{1}^{\mathsf{A}}(N) \leq \mu]. 
\end{equation}
In order to show~\eqref{eq:domination} let $I_N^{\lex}$ and $I_N^{\mathsf{A}}$ denote the sets of isolated vertices in $G_N^{\lex}$ and $G_N^{\mathsf{A}}$, respectively. We will show by induction on $N$ that it is possible to couple $G_N^{\lex}$ and $G_N^{\mathsf{A}}$ such that there is a permutation $\pi_N$ of the vertex set  with the property that $I_N^\lex \subseteq \pi_N(I_N^\mathsf{A}) $; this immediately establishes~\eqref{eq:domination}. 

The claim is trivial for $N = 0$ (with $\pi_0$ being the identity map). For the induction step, let $N \in \mathbb{N}$ and let $\pi_N$ be a permutation with the required property. Let $v_1, \dots, v_\ell$ be the $\ell$ random vertices selected at the beginning of round $N+1$. Then we create $G_{N+1}^\mathsf{A}$ as usual, i.e., by adding to $G_{N}^\mathsf{A}$ the edge that we choose according to $\mathsf{A}$ when $v_1, \dots, v_\ell$ (and $G_N^\mathsf{A}$) are presented. The crucial idea of the coupling is that we may assume that $\Rlex$ is presented the vertices $\pi_N(v_1),\ldots,\pi_N(v_\ell)$. More formally, we create a second graph $G$ that includes all edges in $G^\lex_N$ and an additional edge $e$, which is the edge that $\Rlex$ would choose when presented the images of the $v_i$'s under $\pi_N$. That is, $e = \{\pi_N(v_i), \pi_N(v_{i+1})\}$ and $i = \Rlex(c(\pi_N(v_1)), \dots, c(\pi_N(v_\ell)))$, where $c(u)$ denotes the number of vertices in the component that contains $u$ in $G_N^\lex$. Since the $v_i$'s are uniformly random and $\pi_N$ is a permutation of the vertices  we infer that $G$ is distributed like $G_{N+1}^\lex$, and so this construction is indeed a coupling for $G_N^{\lex}$ and $G_N^{\mathsf{A}}$.

It remains to show the existence of a permutation $\pi_{N+1}$ such that $I_{N+1}^\lex \subseteq \pi_{N+1}(I_{N+1}^\mathsf{A})$. Note that it suffices to show that $|I_{N+1}^\lex| \le |I_{N+1}^\mathsf{A}|$. However, this is a consequence of the fact that $\Rlex$ favours $1$-components. For example, suppose that $\mathsf{A}$ selects the edge $\{u,v\}$ such that $u,v \in I_N^\mathsf{A}$ and moreover, $\pi_N(u), \pi_N(v) \in I_N^{\lex}$. Then both $|I_N^\lex|$ and $|I_N^\mathsf{A}|$ decrease by two (albeit $\Rlex$ might select a different edge joining two isolated vertices in $G^\lex_N$), and the induction hypothesis implies $|I_{N+1}^\lex| \le |I_{N+1}^\mathsf{A}|$. More generally, in the case $u,v \in I_N^\mathsf{A}$ set $s = |\{x \in \{u,v\}: \pi_N(x) \not\in I_N^{\lex}\}| \in \{0,1,2\}$; we just handled the case $s = 0$. Since $I_N^\lex \subseteq \pi_N(I_N^\mathsf{A})$ this definition imples $|I_N^\lex| \le |I_N^\mathsf{A}| - s$. Moreover, $|I_N^\mathsf{A}|$ will decrease by $2$, and $|I_N^\lex|$ will decrease by at least $2-s$. Again the hypothesis guarantees $|I_{N+1}^\lex| \le |I_{N+1}^\mathsf{A}|$. The other cases (i.e., when $u \not\in I_N^\mathsf{A}$ or $v \not\in I_N^\mathsf{A}$) follow by completely analogous arguments, so we leave them as an easy exercise to the reader.

To make use of~\eqref{eq:domination}, recall that Theorem~\ref{thm:main} (b) implies that for every $c\in \RR$,
\begin{align*}
\lim_{n\to\infty}\Pr\left[Y_{1}^\text{lex}\left(\left\lfloor\frac{n\log n
+cn}{\ell}\right\rfloor\right) =0\right] =  e^{-d_1e^{-c}}.
\end{align*}
Thus, by~\eqref{eq:domination} we also have 
\begin{equation*}
\label{eq:generalrules}
\begin{split}
\limsup_{n\to \infty} \Pr\left[\Tcon^\mathsf{A} \leq \frac{n\log n
+cn}{\ell}\right] & \leq
\limsup_{n\to\infty}\Pr\left[Y_{1}^\mathsf{A}\left(\left\lfloor\frac{n\log n
+cn}{\ell}\right\rfloor\right) =0\right]\\ & \leq
\lim_{n\to\infty}\Pr\left[Y_{1}^\text{lex}\left(\left\lfloor\frac{n\log n
+cn}{\ell}\right\rfloor\right) =0\right],
\end{split}
\end{equation*}
and a similar statement follows for the expectation. In this sense, among all
$\ell$-Achlioptas processes, $\Rlex$ is the fastest to connect a graph.

\bibliographystyle{plain}
\bibliography{achlioptas}

\end{document}